\documentclass{amsart}
\usepackage[utf8]{inputenc}
\usepackage[all,cmtip]{xy}
\usepackage{amssymb}
\usepackage{enumerate}
\usepackage{stmaryrd}
\usepackage[hidelinks]{hyperref}

\newcommand{\PP}{\mathbb P}
\newcommand{\CC}{\mathbb C}

\newcommand{\OO}{\mathcal O}

\newcommand{\ch}{\operatorname{ch}}
\newcommand{\op}[1]{\operatorname{#1}}

\newcommand{\todd}{\operatorname{td}}

\newcommand{\Sym}{\operatorname{Sym}}
\newcommand{\Quot}{\operatorname{Quot}}

\newcommand{\ext}{\operatorname{ext}}
\DeclareMathOperator{\Pic}{Pic}

\DeclareFontFamily{OT1}{pzc}{}
\DeclareFontShape{OT1}{pzc}{m}{it}{<-> s * [1.10] pzcmi7t}{}
\DeclareMathAlphabet{\mathpzc}{OT1}{pzc}{m}{it}

\newcommand{\TOR}{\operatorname{\mathpzc{Tor}}}

\newtheorem{lem}{Lemma}[section]
\newtheorem{cor}[lem]{Corollary}
\newtheorem{prop}[lem]{Proposition}
\newtheorem{thm}[lem]{Theorem}
\newtheorem*{thmA}{Theorem A}
\newtheorem*{thmB}{Theorem B}
\newtheorem{conj}[lem]{Conjecture}
\newtheorem*{prop*}{Proposition}
\theoremstyle{definition}
\newtheorem{defn}[lem]{Definition}
\newtheorem{rem}[lem]{Remark}
\newtheorem{ex}[lem]{Example}

\title[Strange duality, quot schemes, and multiple point formulas]{Le Potier's strange duality, quot schemes, and multiple point formulas for del Pezzo surfaces}

\author{Aaron Bertram, Thomas Goller, and Drew Johnson} \thanks{The authors were partially supported by the NSF RTG grant DMS-1246989.}

\begin{document}

\begin{abstract}
We study Le Potier's strange duality on del Pezzo surfaces using quot schemes to construct independent sections of theta line bundles on moduli spaces of sheaves, one of which is the Hilbert scheme of $n$ points. For $n \le 7$, we use multiple point formulas to count the length of the quot scheme, which agrees with the dimension of the space of sections on the Hilbert scheme. When the surface is $\PP^2$ and $n$ is arbitrary, we use nice resolutions of general stable sheaves to show that the quot schemes that arise are finite and reduced. Combining our results, we obtain a lower bound on the rank of the strange duality map, as well as evidence that the map is injective when $n \le 7$.
\end{abstract}

\maketitle

\section{Introduction}
Let $S$ be a smooth, projective del Pezzo surface over the complex numbers with ample anti-canonical 
class $-K_S$. Let
$$ e,f \in H^*(S,{\mathbb Q}) =  H^0(S,{\mathbb Q}) \oplus H^2(S,{\mathbb Q}) \oplus H^4(S,{\mathbb Q}) $$ 
be cohomology classes that are orthogonal with respect to the {\it Mukai pairing}
\[
\chi(e,f) = \int_S e^\vee \cup f \cup \todd(S),
\]
where we write $e = (e_0,e_1,e_2)$ (as with Chern classes), $e^\vee = (e_0,-e_1,e_2)$ and 
$\todd(S) = (1,-{K_S}/2,1)$ is the Todd class of  $S$. This pairing is designed so that 
if $E$ and $F$ are coherent sheaves on $S$, then their Chern characters pair as
$$\chi\big(\ch(E),\ch(F)\big) = \chi(E,F) = \sum (-1)^i\ext^i(E,F) $$
 by the Hirzebruch-Riemann-Roch Theorem. 

Orthogonal classes $e,f$ are {\it candidates for strange duality} if  the moduli spaces $M_S(e^{\vee})$ and $M_S(f)$ of Gieseker-semistable coherent sheaves with the indicated Chern characters
are non-empty, and if the following conditions on pairs $(\hat{E},F) \in M_S(e^{\vee}) \times M_S(f)$ are satisfied:
\begin{enumerate}[(a)]
\item $h^2(\hat{E} \otimes F) = 0$ and $\TOR^1(\hat{E},F)=\TOR^2(\hat{E},F)=0$ for all $(\hat{E},F)$ away from a codimension $\ge 2$ subset, and
\item $h^0(\hat{E} \otimes F)=0$ for some $(\hat{E},F)$.
\end{enumerate}
These conditions ensure that the ``jumping locus''
$$\Theta = \{\, (\hat{E},F) \mid h^0(\hat{E} \otimes F) > 0 \,\} \subset M_S(e^{\vee}) \times M_S(f) $$
has the structure of a Cartier divisor by a classical argument exhibiting $\Theta$ as the vanishing locus of a map of vector bundles of the same rank (\cite{le_potier_dualite}, \cite{MR2319915}, \cite{danila_resultats_2002}). It follows from the del Pezzo condition on $S$ that the 
line bundle associated to $\Theta$ satisfies
$$ {\mathcal O}_{M_S(e^{\vee})\times M_S(f)}(\Theta) = \pi_1^*{\mathcal O}_{M_S(e^{\vee})}(\Theta_f)\otimes \pi_2^*{\mathcal O}_{M_S(f)}(\Theta_{e^{\vee}}),$$
where $\Theta_f$ and $\Theta_{e^{\vee}}$ are restrictions of $\Theta$ to general fibers of the projections. Thus
$$H^0\big(M_S(e^{\vee}) \times M_S(f), {\mathcal O}(\Theta)\big) = H^0\big(M_S(e^{\vee}),{\mathcal O}(\Theta_f)\big) \otimes 
H^0\big(M_S(f),{\mathcal O}(\Theta_{e^{\vee}})\big),$$
so a section defining $\Theta$ determines a map
$$\operatorname{SD}_{e,f}: H^0\big(M_S(f),{\mathcal O}(\Theta_{e^{\vee}})\big)^* \rightarrow H^0\big(M_S(e^{\vee}),{\mathcal O}(\Theta_f)\big)$$
that is well-defined up to a choice of a (non-zero) scalar.

\begin{conj}[Le Potier's Strange Duality] $\operatorname{SD}_{e,f}$ is an isomorphism.
\end{conj}

As a first check (which motivated Witten's analogous conjecture for curves), one might 
ask whether the dimensions of the two vector spaces coincide. Unfortunately, these dimensions are not known, in general,
even when $S = \PP^2$. 
An exception to this is the case $f = (1,0,-n)$, the Chern character of the ideal sheaf ${\mathcal I}_Z \subset {\mathcal O}_S$
of a zero-dimensional length $n$ subscheme $Z \subset S$. In this case, we may identify
$$M_S(f) = S^{[n]}$$
with the Hilbert scheme, and in \S 3, we use \cite{ellingsrud_cobordism_2001} 
to compute $h^0\big(S^{[n]},\OO(\Theta_{e^{\vee}})\big)$ for all orthogonal classes $e$ and $n \le 7$. Interestingly, even here there is 
no known closed formula for the dimensions as functions of $n$ and $e$. 

\medskip

Our strategy in this paper is to attack ``half'' the strange duality conjecture in the case $f = (1,0,-n)$, 
i.e. to show that $\operatorname{SD}_{e,f}$ is injective 
by using a Grothendieck quot scheme argument from \cite{marian_level-rank_2007}. The argument is as follows. 

Let $v = e + f$ be the Chern character of a direct sum $E \oplus F$ of coherent sheaves of Chern characters $e$ and $f$, and suppose 
$V$ is a coherent sheaf with $\ch(V) = v$. Then each element of the Grothendieck quot scheme $\Quot(V,f)$ of coherent sheaf quotients of $V$ of class $f$ gives rise to an exact sequence
\[
	0 \rightarrow E \rightarrow V \rightarrow F \rightarrow 0 \tag{$*$}
\]
and $\chi(E,F) = 0$ is the expected dimension of $\Quot(V,f)$. Moreover, if
$$\hom(E,F) = \ext^1(E,F) = \ext^2(E,F) = 0$$
then the point of $\Quot(V,f)$ corresponding to $(*)$ is isolated and reduced. 

Now suppose a sufficiently general $V$ may be chosen so that $\Quot(V,f)$ is finite and that each quotient
\[
	0 \rightarrow E_i \rightarrow V \rightarrow F_i \rightarrow 0 \tag{$*_i$}
\]
has the property that $E_i$ and $F_i$ are Gieseker-semistable, that $\ext^j(E_i,F_i) = 0$ for all $i,j$, and that $E_i$ is locally free.\footnote{Local freeness will be automatic since $V$ will be locally free and the $F_i$ will be torsion free.} It then follows from stability that $h^0(E_i^* \otimes F_j) = \hom(E_i,F_j) > 0$ for $i \ne j$, and hence that the hyperplanes in $H^0\big(M_S(f),\OO(\Theta_{e^{\vee}})\big)$ determined by the points $F_i \in M_S(f)$ map to linearly independent lines $\Theta_{F_i} \in P\big(H^0(M_S(e^{\vee}),\OO(\Theta_f))\big)$ under the map $\operatorname{SD}_{e,f}$. Thus, the existence of such a sufficiently general $V$ would imply that the length of $\Quot(V,f)$ is bounded above by the dimension of the vector space $H^0\big(M_S(f),\OO(\Theta_{e^{\vee}})\big)$, and that if the dimensions agree, then $\operatorname{SD}_{e,f}$ is injective.

\medskip

This brings us to a second reason for considering the case $f= (1,0,-n)$, namely, that in this case we may find an ``expected'' length of $\Quot\big(V,(1,0,-n)\big)$ by interpreting ideal sheaf quotients of the vector bundle $V$ as multiple points of a map from an auxiliary variety (obtained from $V$) to projective space. Before we get too far into this analysis, we consider several special cases.

\subsection{Torsion sheaves on $\PP^1$.}

Instead of a del Pezzo surface, we consider here the ``del Pezzo curve'' $\PP^1$ and the
orthogonal cohomology classes $e = (0,m)$ and $f = (0,n)$ in $H^*(\PP^1,{\mathbb Z})$. These are the 
Chern characters of 
Gieseker-semistable torsion sheaves ${\mathcal O}_T$ and ${\mathcal O}_U$ associated to subschemes $T,U \subset \PP^1$ which vary in 
moduli spaces $\PP^m = \Sym^m(\PP^1)$ and $\PP^n = \Sym^n(\PP^1)$, respectively. Moreover, 
$$\Theta = \{\, (\OO_T,{\mathcal O}_U) \mid h^0(\OO_T \otimes \OO_U) > 0 \,\} \subset \PP^m \times \PP^n $$
is a divisor of bidegree $(n,m)$, and strange duality in this case is the assertion that
$$\operatorname{SD}_{e,f}: H^0(\Sym^m\PP^1, {\mathcal O}_{\PP^m}(n))^* \rightarrow H^0(\Sym^n\PP^1, {\mathcal O}_{\PP^n}(m))$$
is an isomorphism. 

Now consider a ``general'' coherent sheaf $V$ on $\PP^1$ of class $v = e + f = (0,m+n)$. Here the notion of general is easy to quantify: a general sheaf is any structure sheaf ${\mathcal O}_V$ of a {\it reduced} subscheme $V \subset \PP^1$ of length $m + n$. Then $\Quot({\mathcal O_V},f)$ consists of all choices of $n$ of the $m+n$ points of $V$. This is a reduced scheme of cardinality
$$\binom{m + n}{n} = h^0\big(\Sym^m\PP^1, {\mathcal O}_{\PP^m}(n)\big) = h^0\big(\Sym^n\PP^1, {\mathcal O}_{\PP^n}(m)\big),
$$
which shows, by the argument above, that the strange duality map in this context is an isomorphism.\footnote{An unusual feature of this example is that $e$ is not dualized in the setup of strange duality.} In fact, if $\CC^2$ is the (self-dual) standard representation of SL$(2,\CC)$, then this isomorphism from strange duality is the Hermite reciprocity isomorphism between representations $\Sym^m \Sym^n(\CC^2)$ and 
$\Sym^n \Sym^m(\CC^2)$. 

\subsection{Rank one} Back to the case of del Pezzo surfaces, this is the case originally considered by Le Potier, in which $f = (1,0,-n)$ and $e = (1,-L,s)$ and the condition $\chi(e,f) = 0$ implies $\chi(V^*)=1$
when $V^*$ is the dual of a (rank two) vector bundle $V$ of class $e + f$ on $S$ . Moreover, the second Chern class satisfies
\[
	c_2(V^*) = \chi(L) = h^0(L),
\]
so one expects $V^*$ to have one section vanishing at $\chi(L)$ points. Indeed, let $V^*$ be a general extension $0 \to \OO_S \to V^* \to L \otimes \mathcal{I}_W \to 0$, where $W \subset S$ is a general zero-dimensional subscheme of length $\chi(L)$ and thus $L \otimes \mathcal{I}_W$ has no cohomology. Then $V^*$ has a unique section, which vanishes on $W$.

The two spaces of sections in strange duality are pulled back from rational maps
$$\phi \colon S^{[n]} \rightarrow \operatorname{Gr}\big(\chi(L),n\big) \quad \text{and} \quad 
\psi \colon S^{[m]} \rightarrow \operatorname{Gr}\big(n,\chi(L)\big)$$
to Grassmannians of quotients (respectively subspaces) of dimension $n$ (see \cite{arcara_minimal_2013}). Here $m = \chi(L)-n$ and $M_S(e^{\vee}) \simeq S^{[m]}$. Strange duality reduces to the ordinary duality of sections of ${\mathcal O(1)}$ on the Grassmannians of quotients and subspaces of a fixed dimension of a fixed vector space, as observed originally by Le Potier.

As for the quot scheme, notice that each choice of $n$ of the $\chi(L)$ zeros $W \subset S$ of the section of $V^*$ gives a factorization
$$q: V \twoheadrightarrow {\mathcal I}_W \rightarrow {\mathcal I}_Z,$$ 
where $Z \subset S$ is the reduced scheme supported at the $n$ points. By \emph{tilting} the category of coherent sheaves on $S$, we may regard these ideal sheaves as the quotients of $V$ of class $f = (1,0,-n)$ and notice that the number of them matches the dimensions of the sections of the determinant line bundles associated to $\Theta_{e^{\vee}}$ and $\Theta_f$. The kernel of such a ``surjection'' $q: V \rightarrow \mathcal{I}_Z$ is the {\it derived dual} of a twisted ideal sheaf ${\mathcal I}_{Z'} \otimes L$ for the complementary scheme $Z'$ of $Z \subset W$. 

\medskip

\subsection{Ideal sheaves of a single point.} \label{sec:intro-single-point} Consider the class $f = (1,0,-1)$ of an ideal sheaf ${\mathcal I}_p \subset {\mathcal O}_S$ of a point $p\in S$ and the moduli space $S$ of such sheaves. The classes orthogonal to $f$ have the form
$$e = \big(r,-L,\tfrac{1}{2} K_S.L\big),$$
where $r \in {\mathbb N}$ and $L$ is an ample line bundle on $S$.

If $V$ is a vector bundle of class $v = e + f$, it follows that the dual $V^*$ satisfies
$$\operatorname{rk}(V^*) = r + 1, \quad c_1(V^*) = L, \quad \text{and} \quad \chi(V^*) = r.$$
Choose $V^*$ as a general extension $0 \to \OO_S^r \to V^* \to L \otimes \mathcal{I}_W \to 0$, where $W$ is a general zero-dimensional subscheme of length $\chi(L)$, so $L \otimes \mathcal{I}_W$ has no cohomology. Then $h^0(V^*) = r$ and these $r$ sections compute the second Chern class of $V^*$, namely the global section map $\OO_S^r \rightarrow V^*$ drops rank at $c_2(V^*)$ isolated points, and each such point is the single zero of a single section (up to scalar multiples) of $V^*$.
 
Another computation gives
$$c_2(V^*) = \tfrac{1}{2}L.(L - K_S) + 1 = \chi(L) = h^0(L)$$
and
$L = {\mathcal O}(\Theta_{e^{\vee}})$ is the determinant line bundle on $S = M_S(f)$ in this context.
Each of the $\chi(L)$ distinct sections of $V^*$ vanishing at a point of $S$ dualizes to a quotient $V \twoheadrightarrow {\mathcal I}_{p_i} \subset {\mathcal O}_S$ and the length of $\Quot\big(V,(1,0,-1)\big)$ equals $c_2(V^*) =
\chi(L) = h^0(M_S(f), \Theta_{e^{\vee}})$,
proving in this case that the strange duality map is injective.

\medskip

In the case of general $n$ (for 
the class $f = (1,0,-n)$) and general rank $r$ and ample line bundle $L$ (for $e = (r,-L,s)$, with $s$ determined by $r$ and $L$), we will follow a line of reasoning similar to the special case 1.3 to conclude

\begin{thmA} \label{thm:A} For $1 \le n \le 7$, $r \ge 2$, and $L$ sufficiently ample, the dimension $h^0\big(S^{[n]}, \OO(\Theta_{e^{\vee}})\big)$ agrees with the {\it expected length} of the quot scheme $\Quot\big(V,(1,0,-n)\big)$ of quotients of a vector bundle $V$ on $S$ of class $e + (1,0,-n)$.
\end{thmA}

The proof of Theorem A will unfold across several sections. The case $n=1$ was completed in \S \ref{sec:intro-single-point}. In \S \ref{s:determinant}, we compute $h^0\big(S^{[n]},\OO(\Theta_{e^{\vee}})\big)$ using a generating series of Euler characteristics of line bundles on $S^{[n]}$ provided in \cite{ellingsrud_cobordism_2001}. In \S \ref{s:classical}, we compute the expected length of the quot scheme in the cases $(n,r) \in \{\,(2,2),(3,2),(2,3)\,\}$ using ``classical'' double and triple point formulas of \cite{kleiman_multiple-point_1981}. In \S \ref{s:multiple point formulas}, we perform a similar computation of the expected length of the quot scheme for general pairs $(n,r)$, but here we cannot employ Kleiman's multiple point formulas (\cite{kleiman_multiple-point_1981}) since the relevant maps have corank 2. Instead, we use multiple point formulas obtained by combining results of \cite{kazarian_multisingularities_2003},
\cite{marangell_general_2010}, and \cite{berczi_thom_2012}; a key theorem in \cite{marangell_general_2010} requires the restriction $n \le 7$. Unfortunately, these multiple point formulas (which generalize Kleiman's formulas) are only known to be valid for maps that satisfy certain topological transversality conditions (called ``admissibility'' in \cite{marangell_general_2010}) that have not even been formulated algebraically. The word ``expected'' in the theorem reflects two sources of uncertainty: the admissibility of the map we construct, and the lack of an identification between the $n$-fold point locus of the map and $\op{Quot}\big( V,(1,0,-n) \big)$ that takes into account non-reduced structure. Only the latter is relevant in the three ``classical'' cases.

\medskip

Our second theorem addresses the question begged by the first, namely, the existence of vector bundles $V$ that are sufficiently general so that the quot scheme is in fact finite, reduced, and of the expected length. This a delicate question, even if we focus only on finiteness and reducedness of the quot scheme. Our result is on $S = \PP^2$, but we expect that the conclusion holds on all del Pezzo surfaces.

We consider short exact sequences $0 \to E \to V \to F \to 0$ of sheaves on $\PP^2$ in which $f=\ch(F)=(1,0,-n)$ are the invariants of ideal sheaves $\mathcal{I}_Z$ of $n$ points and $e = \op{ch}(E)$ satisfies $\chi(e,f) = 0$, namely
\[
	e= (r, -\lambda, (n-1)r - \tfrac{3}{2}\lambda)
\]
and thus
\[
v=\op{ch}(V)= e+f = (r+1, -\lambda, (n-1)r-n - \tfrac{3}{2}\lambda).
\]

\begin{thmB}\label{thm:finitequot}
Suppose $n \ge 1$, $r \ge 2$, and $\lambda \gg 0$. Let $V$ be a general stable vector bundle on $\PP^2$ with $\op{ch}(V)=v$ as above. Then the quot scheme $\Quot\big(V,(1,0,-n)\big)$ is finite and reduced, and each quotient is an ideal sheaf of a reduced subscheme.
\end{thmB}

\S \ref{s:genericity} contains the proof of Theorem B, which uses the fact that the duals of general stable vector bundles $V$ with $\ch(V)=v$ have resolutions of the form
\[
	0 \to \OO(-2)^C \to \OO(-1)^B \oplus \OO^A \to V^* \to 0.
\]
Working with resolution spaces instead of moduli of stable sheaves allows us to sidestep questions of stability that arise when studying these quot schemes, such as whether the general quotient $V \twoheadrightarrow \mathcal{I}_Z$ has a stable kernel. We also use resolutions of this form to deduce a statement of general interest about when general stable bundles on $\PP^2$ are globally generated, which we could not find in the literature:
\begin{prop*}[Proposition \ref{p:gg}] Let $\xi=(r,\lambda,d)$ be a Chern character such that $r \ge 1$, $\lambda \ge 0$, and $\chi(\xi) \ge r+2$. Then general sheaves in $M(\xi)$ are globally generated.
\end{prop*}

To relate Theorems A and B, we use the resolutions above to observe that general stable $V^*$ arise as general extensions
\[
	0 \to \OO^r \to V^* \to \OO(\lambda) \otimes \mathcal{I}_W \to 0
\]
for general zero-dimensional subschemes $W$ (Corollary \ref{c:same_V*}). These extensions are exactly the vector bundles constructed in the proof of Theorem A. Thus we can combine Theorems A and B to get

\begin{cor}\label{c:SD_injective} With the hypotheses in Theorem B and $e = \big(r,-\lambda,(n-1)r-\tfrac{3}{2}\lambda\big)$, the rank of the strange duality map
\[
\operatorname{SD}_{e,(1,0,-n)}: H^0\big((\PP^2)^{[n]}, \OO(\Theta_{e^{\vee}})\big)^* \rightarrow H^0\big(M(e^{\vee}), \OO(\Theta_{(1,0,-n)})\big)
\]
is bounded below by the length of $\op{Quot}\big(V,(1,0,-n)\big)$. Moreover,
\begin{enumerate}[(a)]
\item For $n=1$ and the ``classical'' cases $(n,r) \in \{\, (2,2),(2,3),(3,2)  \,\}$, $\operatorname{SD}_{e,(1,0,-n)}$ is injective;
\item When $n \le 7$, injectivity of $\op{SD}_{e,(1,0,-n)}$ is predicted by Theorem A.
\end{enumerate}
\end{cor}

In \S6, we provide references to other strange duality results on surfaces and speculate about possible ways to generalize the results in this paper. In particular, we would like to conclude that the strange duality map is injective for $n \le 7$. This would require either more rigorous multiple point formulas or some other method for computing lengths of finite quot schemes. We are also not satisfied with only ``half'' of the strange duality conjecture. The other half would follow from a computation of $h^0\big(M(e^{\vee}),\OO(\Theta_{(1,0,-n)})\big)$, but this is difficult since the moduli space $M(e^{\vee})$ is more mysterious than the Hilbert scheme of points.

\section{Determinant Bundles on Hilbert Schemes}\label{s:determinant}
When $f=(1,0,-n)$ and $M_S(f)$ is the Hilbert scheme $S^{[n]}$, we can identify the determinant line bundle $\OO(\Theta_{e^{\vee}})$ explicitly. We first recall the standard description of the Picard group of $S^{[n]}$. Given a line bundle $L$ on $S$, one constructs a line bundle $L_n$ on $S^{[n]}$ as follows. The line bundle $\bigotimes_{i=1}^n \pi_i^*(L)$ on $S^n$ (where $\pi_i \colon S^n \rightarrow S$ is the $i$th projection) comes equipped with an action of the permutation group $\mathfrak S_n$, so it descends to a line bundle on the symmetric product $S^{(n)}$, which can be pulled back to $S^{[n]}$ using the Hilbert-Chow morphism.
For an effective divisor $D$, $D_n$ can be thought of as the locus of subschemes whose support meets $D$. This process gives a map $\Pic(S) \rightarrow \Pic(S^{[n]})$ that induces an isomorphism
\[
 \Pic(S) \oplus \mathbb Z \left[\mathcal O_{S^{[n]}}\left(\tfrac{B}{2}\right)\right] \cong \Pic(S^{[n]}),  
\]
where $B$ is the exceptional divisor of the Hilbert-Chow morphism.

We can parametrize Chern characters $e$ orthogonal to $f = (1,0,-n)$ as
\[
	e = (r, -L, (n-1)r + \tfrac{1}{2}L.K_S).
\]
A sheaf $\hat{E}$ on $S$ with $\ch(\hat{E})=e^{\vee}$ induces a determinant line bundle on $S^{[n]}$ whose class is 
\[
\Theta_{\hat{E}}=-c_1\big(Rq_*(p^*\hat{E} \otimes \mathcal{I}_{\mathcal Z})\big) = c_1(\hat{E})_n - \op{rk}(\hat{E}) \tfrac{B}{2} = L_n - r \tfrac{B}{2},
\]
where $S^{[n]} \xleftarrow{q} S^{[n]} \times S \xrightarrow{p} S$ are the projections, $\mathcal{Z} \subset S^{[n]} \times S$ is the universal subscheme, and the explicit formula on the right is obtained by Grothendieck-Riemann-Roch (see Lemma 3.5 of \cite{wandel_stability_2013} or \cite{ellingsrud_cobordism_2001} \S5). Since the Picard group of a del Pezzo surface is discrete, the line bundle determined by $\Theta_{\hat{E}}$ depends only on $e^{\vee} = \op{ch}(\hat{E})$, so it makes sense to write
\[
	\OO(\Theta_{e^{\vee}}) = \OO_{S^{[n]}}(L_n - r\tfrac{B}{2}).
\]

To investigate strange duality for Hilbert schemes of surfaces, we need to compute the number of sections of $\OO(\Theta_{e^{\vee}})$. The Euler characteristic can be computed using
\begin{thm}[\cite{ellingsrud_cobordism_2001}, Theorem 5.3] \label{thm:chi_gen}
For any surface $S$,
\[
 \sum_{n\ge 0} \chi\left(\OO_{S^{[n]}}(L_n - r\tfrac{B}{2})\right)z^n = g_r(z)^{\chi(L)} \cdot f_r(z)^{\tfrac{1}{2}\chi(\OO_S)} \cdot A_r(z)^{L.K_S-\tfrac{1}{2}K_S^2}\cdot B_r(z)^{K_S^2},
\]
where $A_r(Z)$, $B_r(Z)$, $f_r(z)$, $g_r(z)$ are power series in $z$ depending only on $r$, and
\begin{align*}
f_r(z) &= \sum_{k\ge0} \binom{(1-r^2)(k-1)}{ k} z^k \\
g_r(z) &= \sum_{k\ge0} \frac{1}{1-(r^2-1)k}\binom{1-(r^2-1)k}{k} z^k.
\end{align*}
\end{thm}

Explicit formulas for $A_r(z)$ and $B_r(z)$ are not known, but as explained in \cite{ellingsrud_cobordism_2001}, one can use the localization techniques of \cite{ellingsrud_botts_1996} and \cite{ellingsrud_homology_1987} to determine the first few terms of the power series on the left side for $\PP^2$ and any $r$ and $L$. Substituting appropriate choices of $L$, one can solve for the first few terms of $A_r(z)$ and $B_r(z)$. In \cite{ellingsrud_cobordism_2001}, these power series are computed up to order 5. This was not sufficient for our purposes, so we implemented the suggested computations in Sage
(\cite{developers_sagemath_2015}). We used up to order 7 in this paper (although our code can compute more):
\begin{align*}
A_r(z) &= 1 + \left(-\tfrac{1}{6} r^{3} + \tfrac{1}{6} r\right)z^{2} +
\left(\tfrac{17}{40} r^{5} - \tfrac{5}{8} r^{3} + \tfrac{1}{5}
r\right)z^{3} + 
\left(-\tfrac{631}{630} r^{7} + \tfrac{1}{72} r^{6} \right.\\
&+\left.\tfrac{88}{45} r^{5} - \tfrac{1}{36} r^{4} - \tfrac{209}{180} r^{3} +
\tfrac{1}{72} r^{2} + \tfrac{29}{140} r\right)z^{4} +
\left(\tfrac{171215}{72576} r^{9} - \tfrac{17}{240} r^{8} -
\tfrac{69619}{12096} r^{7} \right.\\
&\left.+ \tfrac{7}{40} r^{6} + \tfrac{16979}{3456}
r^{5} - \tfrac{11}{80} r^{4} - \tfrac{31259}{18144} r^{3} + \tfrac{1}{30}
r^{2} + \tfrac{13}{63} r\right)z^{5} +
\left(-\tfrac{18684667}{3326400}
r^{11} \right.\\
&\left.+ \tfrac{155581}{604800} r^{10} + \tfrac{597209}{36288} r^{9} -
\tfrac{1699}{2240} r^{8} - \tfrac{5513891}{302400} r^{7} +
\tfrac{23033}{28800} r^{6} + \tfrac{114685}{12096} r^{5} \right.\\
&\left.-\tfrac{2669}{7560} r^{4} - \tfrac{519509}{226800} r^{3} + \tfrac{229}{4200}
r^{2} + \tfrac{281}{1386} r\right)z^{6} + \left(\tfrac{401297449}{29652480} r^{13} - \tfrac{8914439}{10886400}
r^{12} \right.
\\ &\left.- \tfrac{528153667}{11404800} r^{11} + \tfrac{30585833}{10886400}
r^{10} + \tfrac{4352347}{69120} r^{9} - \tfrac{13405099}{3628800} r^{8} -\tfrac{44899771}{1036800} r^{7} \right. \\
&\left.+ \tfrac{25156259}{10886400} r^{6} +
\tfrac{817639}{51840} r^{5} - \tfrac{1859441}{2721600} r^{4} -
\tfrac{339287}{118800} r^{3} + \tfrac{1433}{18900} r^{2} + \tfrac{85}{429}
r\right)z^{7} + \cdots \\ \vspace{.5cm}
B_r(z) &= 1 + \left(-\tfrac{1}{24} r^{4} + \tfrac{1}{24} r^{2}\right)z^{2} +
\left(\tfrac{97}{720} r^{6} - \tfrac{31}{144} r^{4} + \tfrac{29}{360}
r^{2}\right)z^{3} + \\
&\left(-\tfrac{14899}{40320} r^{8} + \tfrac{2273}{2880}
r^{6} - \tfrac{3053}{5760} r^{4} + \tfrac{139}{1260} r^{2}\right)z^{4} + \\
&\left(\tfrac{503377}{518400} r^{10} - \tfrac{311701}{120960} r^{8} +
\tfrac{421267}{172800} r^{6} - \tfrac{6257}{6480} r^{4} + \tfrac{187}{1400}
r^{2}\right)z^{5} + \\
&\left(-\tfrac{1205178661}{479001600} r^{12} +
\tfrac{346550543}{43545600} r^{10} - \tfrac{19975933}{2073600} r^{8} +
\tfrac{241348529}{43545600} r^{6} - \tfrac{4092191}{2721600} r^{4} \right. \\
&\left.+
\tfrac{9047}{59400} r^{2}\right)z^{6} + \left(\tfrac{1571744023}{242161920} r^{14} -
\tfrac{11403389887}{479001600} r^{12} + \tfrac{1523544803}{43545600}
r^{10} \right. \\
&\left.- \tfrac{2666500579}{101606400} r^{8} + \tfrac{458713229}{43545600}
r^{6} - \tfrac{63757807}{29937600} r^{4} + \tfrac{6349289}{37837800}
r^{2}\right)z^{7} + \cdots
\end{align*}
As far as we are aware, it remains an open problem to determine whether closed formulas for these series exist.

With these series in hand and a knowledge of the intersection theory on del Pezzo surfaces, it is now straightforward to extract the Euler characteristic of any line bundle on $S^{[n]}$. For example, the formula for $n=2$ is
\[
\chi\big(\OO_{S^{[2]}}(L_2 - r\tfrac{B}{2})\big) = \binom{\chi(L)}{2} - (r^2-1)\chi(L) - \binom{r+1}{3}L.K_S - \binom{r+1}{4}K_S^2 + \tfrac{1}{2}\binom{r^2}{2}
\]
and the formula for $n=3$ and $r=2$ is
\[
	\chi\big(\OO_{S^{[3]}}(L_3 - B)\big) = \binom{\chi(L)}{3} - \chi(L)(3 \chi(L)+L.K_S-21)+9L.K_S+K_S^2-28.
\]

Comparing a result of Bertram and Coskun (\cite{bertram_birational_2013-1} Theorem 2.4) on the nef cone of a del Pezzo surface with the conditions in \cite{di_rocco_kvery_1996} for a line bundle on a del Pezzo surface to be $k$-very ample shows
\begin{prop} $\Theta_{e^{\vee}}$ is nef if and only if $L$ is $(n-1)r$-very ample.\footnote{When $S$ is the blow-up of $\PP^2$ at 8 points, the $(n-1)r$-very ample condition is not quite sufficient since the anticanonical curve imposes an additional condition on the nef cone of $S^{[n]}$ (see \cite{bertram_birational_2013-1} Remark 2 or \cite{bolognese_nef_2016}).}
\end{prop} 
Moreover, subtracting $K_{S^{[n]}} = (K_S)_n$ from a nef $\Theta_{e^{\vee}}$ makes the inequalities defining the nef cone in \cite{bertram_birational_2013-1} strict, namely $\Theta_{e^{\vee}} - K_{S^{[n]}}$ is ample. Thus Kodaira Vanishing guarantees  that the sections of $\OO(\Theta_{e^{\vee}})$ are being counted by the Euler characteristics in the power series above:

\begin{cor}\label{c:euler_counts_sections}
If $L$ is $(n-1)r$-very ample, then
\[
h^0\big(\OO_{S^{[n]}}(L_n - r\tfrac {B}{2})\big) = \chi\big(\OO_{S^{[n]}}(L_n - r\tfrac{B}{2})\big).
\]
\end{cor}



\section{Three ``classical'' cases}\label{s:classical}
In this section, we prove Theorem A in each of the special cases when $(n,r) \in \{\,(2,2),(2,3),(3,2) \,\}$ by constructing a general vector bundle $V$, computing the double or triple points of a particular morphism, and checking that the result agrees with the formula for $\chi\big(\OO_{S^{[n]}}(L_n-\tfrac{r}{2}B)\big)$ from the previous section.

\subsection{Double points of an immersed plane curve, $(n,r)=(2,2)$}

Let $L$ be an ample line bundle on $S$ satisfying $-L.K_S \ge 4$. Let $W$ be a general collection of $|W|=\chi(L)-1$ points in $S$, such that the unique curve $C$ of class $L$ containing $W$ is smooth and $W$ is general on $C$. Let $V^*$ be a general extension
\[
	0 \to \OO_S^3 \to V^* \to L \otimes \OO_C(-W) \to 0.
\]
Then $V^*$ is locally free, has rank 3, and its 3 sections drop to rank 2 on $C$. Thus we get a morphism
\[\xymatrix{
	f \colon C \ar[r] & P \left( H^0(V^*) \right) = \PP^2 
}\]
which sends each point $p \in C$ to the unique (up to scaling) section of $V^*$ vanishing at $p$.

We claim that this morphism is a general projection of the embedding defined by the line bundle $\OO_C(W)$. (The condition $-L.K_S \ge 4$ ensures that $W$ is very ample on $C$ since it is general of degree $\chi(L)-1 \ge \chi(L+K_S) + 3 = g_C+3$.) To identify the sections of $V^*$ that vanish at points, we restrict the sequence defining $V^*$ to $C$. This identifies the sections that vanish as the kernel
\[
	0 \to \OO_C(-W) \to H^0(V^*) \otimes \OO_C \to V^*|_C \to L|_C \otimes \OO_C(-W) \to 0,
\]
so $f$ is induced by the morphism $H^0(V^*)^* \otimes \OO_C \twoheadrightarrow \OO_C(W)$. To ensure that $V^*$ does not contain $\OO_S$-summands, and since automorphisms of $\OO_S^3$ do not affect the isomorphism class of the extension $V^*$, we should choose the extension $V^*$ as a point of $\op{Gr}\big(3,\op{Ext}^1(L \otimes \OO_C(-W),\OO_S)\big)$, and this Grassmannian is naturally isomorphic to the Grassmannian $\op{Gr}\big(3,H^0(\OO_C(W))\big)$ parametrizing choices of three sections of $\OO_C(W)$. Choosing $V^*$ to be a general extension ensures that the sections of $H^0(\OO_C(W))$ are general, which proves the claim.

Thus we see that $f(C)$ is a plane curve with only simple nodes. The preimage $f^{-1}(s)$ of a closed point $s$ is equal to the vanishing locus of $s$ as a section of $V^*$, and since $n=2$, we want to count sections that vanish at exactly two points, which correspond to the nodes of $f(C)$. Since degree is preserved by projection, $f(C)$ has degree $|W| = \chi(L)-1 = c_2(V^*)$ and its normalization $C$ has genus $\tfrac{1}{2}L.(L+K_S)+1$ by adjunction. Thus the number of nodes is the difference
\[
	\binom{c_2(V^*)-1}{2} - \left( \tfrac{1}{2}L.(L+K_S) + 1 \right),
\]
which agrees with $\chi\big(\OO_{S^{[2]}}(L_2-B)\big)$.

\subsection{Double points of a blow-up of $S$ immersed in $\PP^4$, $(n,r) = (2,3)$}

Let $L$ be a very ample line bundle on $S$ satisfying $-L.K_S \ge 5$ and an additional positivity condition that we will state below. Choose $W'$ to be a collection of $\chi(L+K_S)$ general points on $S$ not contained on any curve of class $L+K_S$, which also impose independent conditions on curves of class $L$. Let $C,C'$ be general smooth curves of class $L$ containing $W'$ and intersecting transversally, and let $W \subset C \cap C'$ with $|W| = \chi(L)-2$ be the residual of $W'$. Then $W$ is not contained on a curve of class $L+K_S$ and imposes independent conditions on curves of class $L$ by Cayley-Bacharach (\cite{tan_note_1999}).

Let $V^*$ be a general extension
\[
	0 \to \OO_S^3 \to V^* \to L \otimes \mathcal{I}_W \to 0,
\]
chosen as a point of $\op{Gr}\big(3,\op{Ext}^1(L \otimes \mathcal{I}_W,\OO_S)\big)$ to ensure that $V^*$ has no $\OO_S$-summands. This Grassmannian is nonempty for $-L.K_S \ge 5$. Then $V^*$ is rank 4 with 5 sections, so there is a section vanishing at each point of $S$, and in fact the sections of $V^*$ drop to rank 3 on $W'$. We can separate these additional sections by passing to $X = \op{Bl}_{W'}S \xrightarrow{\pi} S$, which yields a morphism
\[\xymatrix{
	f \colon X = \op{Bl}_{W'}S \ar[r] & P \left( H^0(V^*) \right) = \PP^4
}\]
sending each $p \in X$ to the unique section of $V^*$ vanishing at $p$.

More precisely, $f$ is a general projection to $\PP^4$ of the embedding determined by the line bundle $\mathcal{L}=\pi^*L(-\sum E_i)$, where $E_i$ are the exceptional divisors of the blow-up $\pi$. To see that $\mathcal{L}$ is very ample on $X$, we use a criterion of Ballico-Coppens, Theorem 0.1 in \cite{ballico_very_1997}. First, we view $X$ as a blow-up $\tilde{\pi} \colon X \to \PP^2$ with exceptional divisors $E_i$ from blowing up $W'$ and $F_i$ from the del Pezzo surface $S$. (The case $S=\PP^1 \times \PP^1$ can be handled by a similar argument since blowing up a point of $\PP^1 \times \PP^1$ yields the blow-up of $\PP^2$ at two points.) Write
\[
	\mathcal{L} = \tilde{\pi}^*(\OO_{\PP^2}(d)) \otimes \OO_X\big(- \textstyle{\sum} m_i F_i - \textstyle{\sum} E_i\big).
\]
To apply the criterion, we note that $\OO_{\PP^2}(1)$ is very ample and that $m_i + m_j \le d-1$ (this is property (C1) since our blown-up points are general) since $L$ is very ample on $S$ \cite{di_rocco_kvery_1996}, so the last condition we need to check to guarantee that $\mathcal{L}$ is very ample is
\[	h^1(\PP^2,\mathcal{I}_{\textbf{m}} \otimes \OO_{\PP^2}(d-1))=0,
\]
where $\textbf{m}= \sum m_i P_i + \sum Q_i$ is the weighed sum of the points $P_i,Q_i \in \PP^2$ corresponding to the exceptional divisors $F_i,E_i$. Since all the blown up points are general, this last condition holds if we assume that
\[
	\binom{d+1}{2} - \sum \binom{m_i+1}{2} - |W'| = 2d-1-\sum m_i \ge 0,
\]
which is the positivity hypothesis on $L$ we mentioned above.

To see that $f$ is a general projection to $\PP^4$ of the embedding determined by $\mathcal{L}$, note that the kernel $L^*$ in the exact sequence
\[
	0 \to L^* \to H^0(V^*) \otimes \OO_S \to V^* \to \OO_{W'} \to 0
\]
fails to identify the additional sections of $V^*$ that vanish along $W'$, but this is corrected by pulling back $H^0(V^*) \otimes \OO_S \to V^* \to \OO_{W'}$ to $X$, which yields a new kernel
\[
	0 \to \mathcal{L}^* \to H^0(V^*) \otimes \OO_X \to \pi^* V^* \to \OO_{\sqcup E_i} \to 0,
\]
where $\mathcal{L}$ is defined as above. We can think of $f$ as the induced morphism $P(\mathcal{L}^*) \to P(H^0(V^*))$, which is the composition of the morphism $X \to \PP H^0(\mathcal{L})$ and the projection onto the image of the induced inclusion $H^0(V^*)^* \hookrightarrow H^0(\mathcal{L})$, which by construction contains the span of $C,C'$ viewed as sections of $\mathcal{L}$. For fixed $L$ and $W$, assigning this image to each extension $V^*$ gives an isomorphism
\[ 	\op{Gr}\big(3,\op{Ext}^1(L \otimes \mathcal{I}_W,\OO_S)\big) \simeq \op{Gr}\big(3,H^0(\mathcal{L})/\op{span}(C,C')\big),
\]
and these Grassmannians are non-empty due to $-L.K_S \ge 5$. Thus a general choice of $V^*$ yields a choice of 3 general sections of $\mathcal{L}$ in addition to $C,C'$. Since $C,C'$ are general curves containing $W'$, the sections of a general $V^*$ yield 5 general sections of $\mathcal{L}$, namely the projection in the definition of $f$ is general. See \ref{ss: projective bundle} for a more detailed argument in a similar situation.

Thus $f(X)$ is an immersed surface in $\PP^4$ with ordinary double points. The number of ordinary double points of an immersion can be computed using the Herbert-Ronga formula
\begin{thm}[\cite{kleiman_multiple-point_1981} Theorem 5.8]\label{H-R} Let $X$ and $Y$ be smooth varieties, let $k \ge 2$, and let $f \colon X \to Y$ be practically $k$-generic of codimension $\ell \ge 1$. Suppose $f$ is an immersion. Then the $k$-fold point class in $X$ is
\[
  x_k = f^* f_*x_{k-1} - (k-1)\, c_\ell \, x_{k-1} \;\in A^{(k-1)\ell}(X),
\]
where $c_{\ell}=c_{\ell}(f^* T_Y/T_X)$ and $x_1=[X]$ is the fundamental class.
\end{thm}

In \cite{kleiman_multiple-point_1981}, Kleiman constructs ``derived maps" $f_{k-1} \colon X_k \to X_{k-1}$ inductively ($f_0=f$), where $X_k$ is the residual to the diagonal subscheme in the fibered product of two copies of $f_{k-2}$, and $f_{k-1}$ is defined to be the second projection. The $k$-fold point class $x_k$ is $(f_1 \circ \cdots \circ f_{k-1})_*[X_k]$. In general, $x_k$ can be thought of as the closure of the locus of points $p$ in $X$ at which the fiber of $f(p)$ has length $k$. The practically $k$-generic assumption is that $f_{j-1}$ is an lci of codimension $\ell$ for all $j \le k$, which guarantees that each $x_j$ has the expected codimension $(j-1)\ell$.

For our application, we consider the Herbert-Ronga double point formula
\[
	x_2 = f^* f_*[X]-c_{\ell} \; \in A^\ell(X),
\]
which holds even if $f$ is not an immersion (\cite{kleiman_multiple-point_1981} Theorem 5.6). We want to count certain sections of $V^*$, namely the double points of $f$ in $Y=\PP^4$, so we push forward the Herbert-Ronga formula, dividing by 2 to account for the fact that every double point has two preimages. The resulting formula is
\begin{cor}[Double point formula]\label{double point formula} Let $f \colon X \to Y$ be a morphism of smooth varieties. Assume $f$ is practically $2$-generic and has codimension $\ell = \dim Y - \dim X \ge 1$. Then the double point class in $Y$ is
\[
	y_2 = \tfrac{1}{2} \left( (f_*[X])^2 - f_* c_\ell \right) \; \in A^{2\ell}(X).
\]
\end{cor}
In our application, the first derived map $f_1$ is the second projection to $X$ from the finite locus $X_2 = \{ (p,q) \colon \text{$f(p)=f(q)$ and $p \ne q$} \} \subset X \times X \setminus \Delta$, so $f$ is practically 2-generic and the double point formula correctly computes the number of double points on $f(X)$. To compute $c_2 = c_2(f^*T_{\PP^4}/T_X)$, we use the total Chern classes $c(T_{\PP^4})=(1+H)^5$ and $c(T_X)=1-K_X  + (12\rho-K_X^2)$, where $\rho$ is the class of a point and $K_X = K_S + \sum E_i$, as well as the identity $f^*H = L-\sum E_i$. After some simplification, we get
\[
	c_2 = [(1+(L-\textstyle{\sum} E_i))^5 (1+K_X + 2(K_X^2 - 6))]_2
    = 5L.K_S + 2K_S^2 + 10L^2-7|W'|-12,
\]
and substituting $|W'| = L^2 - c_2(V^*)$ yields
\[
	y_2 = \tfrac{1}{2}\left( c_2(V^*)^2 -7c_2(V^*) - 5L.K_S -2K_S^2-3L^2+12\right)
\]
which matches the above formula for $\chi\left(\OO_{S^{[2]}}(L_2 - 3\tfrac{B}{2})\right)$ when we make the further substitutions $c_2(V^*)=\chi(L)-2$ and $L^2 = 2\chi(L)+L.K_S-2$.

\begin{rem} We could also use the double point formula to recover our count in the above case $(n,r)=(2,2)$ of a degree $c_2(V^*)$ immersion $f \colon C \to \mathbb{P}^2$ of a smooth curve of genus $\tfrac{1}{2}L.(L+K_S)$. In this case $c(T_{\PP^2})=(1+H)^3$ and $c(T_C)=1-K_C$, so 
\[
	c_1=c_1(f^*T_{\PP^2}/T_C) = [f^*(1+3H)(1+K_C)]_1 = 3c_2(V^*)+L.(L+K_S)
\]
and therefore
\[
	y_2 = \tfrac{1}{2}\big( c_2(V^*)^2 - (3c_2(V^*)+L.(L+K_S)) \big),
\]
which agrees with the previous computation.
\end{rem}

\subsection{Triple points of a non-immersed blow-up of $S$ in $\PP^3$, $(n,r)=(3,2)$}

As in the case $(n,r)=(2,3)$, choose sufficiently positive $L$ (with the same conditions except that $-L.K_S \ge 4$ suffices in this case), general $W'$ of length $|W'|=\chi(L+K_S)$, and smooth transversal curves $C,C'$ of class $L$ containing $W'$. Let $W$ be the residual to $W'$ in $C \cap C'$. We let $V^*$ be a general extension $0 \to \OO_S^{\oplus 2} \to V^* \to L \otimes \mathcal{I}_W \to 0$, which has rank 3 and 4 sections that drop to rank 2 on $W'$. As before we get a morphism
\[\xymatrix{
	f \colon X = \op{Bl}_{W'}S \ar[r] & P\left( H^0(V^*) \right) = \PP^3
}\]
sending $p \in X$ to the unique section of $V^*$ vanishing at $p$, and once again $f$ is a general projection to $\PP^3$ of the embedding of $X$ determined by the line bundle $\pi^* L(-\sum E_i)$, where $E_i$ are the exceptional divisors of the blow-up.

Since $f$ is a general projection to $\PP^3$ of an embedded surface, we can give an explicit description of its singularities, following \cite{ciliberto_branch_2011} and \cite{MR1614780}. The singular points of the image $f(X)$ form an irreducible curve $C_0$ (the double point locus) containing finitely many ordinary triple points (which are three-branch nodes of the curve) and finitely many pinch points (which are smooth points of the curve, but at which the derivative of $f$ drops rank by 1, so $f$ is not an immersion). The preimage $C_1 := f^{-1}(C_0) \subset X$ is a curve and $f|_{C_1} \colon C_1 \to C_0$ is generically two-to-one. The pinch points on $C_0$ are branch points of $f|_{C_1}$, over which $C_1$ is smooth. The only singularities of $C_1$ are triples of simple nodes lying over each triple point of $C_0$.

Since $n=3$, we want to compute the number of these triple points. This can be done using Kleiman's triple point formula.
\begin{thm}[\cite{kleiman_multiple-point_1981}, Theorem 5.9]\label{Kleiman triple} If $f \colon X \to Y$ is practically 3-generic of codimension $\ell$ between smooth varieties, then the triple point class in $X$ is
\[
  x_3 = f^* f_* x_2 - 2c_\ell x_2 + \left(\sum_{k=1}^{\ell} 2^k c_{\ell-k} c_{\ell+k}\right) \; \in A^{2\ell}(X).
\]
\end{thm}
To obtain a corresponding formula on $Y$, we substitute the double point formula for $x_2$, push forward to $Y$, divide by 3, and use the projection formula.
\begin{cor}[Triple point formula]\label{triple pt on Y} Let $f \colon X \to Y$ be a practically 3-generic codimension $\ell$ morphism of smooth varieties. Then the triple point class in $Y$ is
\[
	y_3 = \frac{1}{6} \left( (f_*[X])^3 - 3(f_* c_\ell)(f_*[X])+2f_*(c_\ell^2) + \sum_{k=1}^{\ell}2^k f_*(c_{\ell-k} c_{\ell+k}) \right) \; \in A^{3\ell}(Y).
\]
\end{cor}

In our setting, the first derived map $f_1 \colon X_2 \to X_1 = X$ is the normalization of $C_1$, and the set of closed points in $X_2$ lying over the nodes of $C_1$ is exactly the image of the second derived map $f_2 \colon X_3 \to X_2$. Thus $f$ is practically 3-generic, so the triple point formula applies. Note that there are three nodes over each triple point of $C_0$, and two points in $X_3$ over each of these nodes, which explains the factor of $\frac{1}{6}$ in the formula.  

Letting $\rho$ denote the point class in $A^2(X)$, we compute the total Chern class
\[
	c(f^*T_{\PP^3}/T_X) = (1+(L-\textstyle{\sum} E_i))^4(1+K_X + 2(K_X^2-6\rho)),
\]
from which we extract $c_1 = 4L + K_S - 3 \sum E_i$ and $c_2 = 6L^2+4L.K_S+ 2K_S^2-4|W'|-12$ after substituting $K_X = K_S+\sum E_i$. To compute $y_3$ we need to know how a divisor $D$ on $X$ pushes forward under $f$, but this is easy since $f_* D$ in $\PP^3$ is determined by its degree $(f_* D).H = D.f^*H = D.(L-\textstyle{\sum} E_i)$.
In particular $f_* c_1 = 4L^2 + L.K_S -3|W'|$, so
\begin{align*}
	y_3 &= \frac{1}{6} \left[ c_2(V^*)^3 - 3c_2(V^*) f_*(c_1) + 2f_*(c_1^2)+2f_*(c_2) \right] \\
    &= \frac{1}{6} \left[ c_2(V^*)^3-3c_2(V^*)(4L^2+L.K_S-3|W'|)+44L^2+24L.K_S+6K_S^2 \right. \\
 &\phantom{indent}\left.-26|W'|-24 \right],
\end{align*}
which agrees with the above formula for $\chi\big(\OO_{S^{[3]}}(L_3 - \tfrac{B}{2})\big)$ when we substitute $c_2(V^*)=\chi(L)-2$, $|W'| = \chi(L)+L.K_S$, and $L^2 = 2\chi(L)+L.K_S-2$.

\section{Multiple point formulas}\label{s:multiple point formulas}
In this section we prove Theorem A in the general case, namely when $n,r \ge 2$, $(n,r) \notin \{\, (2,2),(2,3),(3,2) \,\}$, and $n \le 7$. To do so, we will choose a general globally-generated vector bundle $V^*$ with the appropriate invariants and collect the sections of $V^*$ that vanish at points as the kernel $0 \to G \to H^0(V^*) \otimes \OO_S \to V^* \to 0$. The sections of $V^*$ that vanish at $n$ points will correspond to the $n$-fold points of the natural map $f \colon P(G) \to P\big(H^0(V^*)\big)$, which on closed points is just $(p,s) \mapsto s$, where $p \in S$ and $s$ is a section of $V^*$ vanishing at $p$. We will compute the number of $n$-fold points of $f$ using multiple point formulas, which are only known in sufficient generality ($f$ has corank 2) up to $n=7$. Our computer code checks that these computations agree with the value of $\chi\big(\OO_{S^{[n]}}(L_n-r\tfrac{B}{2})\big)$ obtained from the power series in \S\ref{s:determinant}.

\subsection{Choosing $V$}\label{ss:choosingV}

Let $L$ be an ample line bundle on $S$. Let $W$ be a collection of $|W| = \chi(L)-(n-1)(r-1)$ distinct points on $S$ satisfying the following genericity condition: the points should impose independent conditions on curves of class $L$ and should not be contained on a curve of class $L+K_S$. Then we define $V^*$ as an extension
\[
	0 \to \OO_S^r \to V^* \to L \otimes \mathcal{I}_W \to 0,
\]
corresponding to a general point in $\op{Gr}\big(r,\op{Ext}(L \otimes \mathcal{I}_W,\OO_S) \big)$, which is non-empty if and only if $-L.K_S \ge n(r-1)+1$. The Grassmannian is the natural extension space since we are mainly interested in the isomorphism class of the middle object in the extension. More precisely, the isomorphism
\[
	\op{Ext}^1(L \otimes \mathcal{I}_W,\OO_S^r) \to \op{Ext}^1(L \otimes \mathcal{I}_W,\OO_S)^r
\]
defined by pushing forward extensions along the $r$ projection maps $p_i \colon \OO_S^r \to \OO_S$ is $\op{GL}(r,\CC)$-equivariant, where the action on the left is by pushing forward along automorphisms of $\OO_S^r$ (which has the effect of pre-composing the map $\OO_S^r \to V^*$ by the inverse automorphism) and the action on the right is the natural action on the $r$ summands. Removing the locus where the action is not free (which corresponds on the left to extensions that have $\OO_S$-summands and on the right to linearly dependent $r$-tuples) and passing to the quotient yields the Grassmannian above.

\begin{prop}\label{gen extensions} Let $n,r \ge 2$ and suppose $-L.K_S \ge 2 + (n+1)(r-1)$ and $L$ is $N=\op{max}\big\{(n-1)(r-1),3 \big\}$-very ample.\footnote{Compare this to the positivity condition on $L$ in Corollary \ref{c:euler_counts_sections}.} Then the extensions $0 \to \OO_S^r \to V^* \to L \otimes \mathcal{I}_W \to 0$ parametrized by general points of $\op{Gr}\big(r,\op{Ext}(L \otimes \mathcal{I}_W,\OO_S) \big)$ satisfy
\begin{enumerate}[(a)]
\item $\ch(V^*)=(r+1,L,(n-1)(r-1)-1+\tfrac{1}{2}LK_S)$;
\item $V^*$ is globally generated $\iff$ $(n,r) \notin \{ (2,2),(2,3),(3,2) \}$;
\item $h^1(V^*)=h^2(V^*)=0$ and $h^0(V^*)=\chi(V^*)=n(r-1)+1$;
\item $V^*$ is locally free;
\item $h^0(V)=0$.
\item No section of $V^*$ vanishes along a curve.
\end{enumerate}
\end{prop}

\begin{proof} Part (a) is additivity of the Chern character. Part (b) follows from the fact that $V^*$ is globally generated if and only if $L \otimes \mathcal{I}_W$ is globally generated. Part (c) is obtained by computing the cohomology long exact sequence and using genericity of $W$. Part (d) holds since the Cayley-Bacharach property is satisfied for $L \otimes \mathcal{I}_W$ (\cite{huybrechts_geometry_2010} Theorem 5.1.1). For (e) one can check the equivalent assertion $\hom(V^*,\OO_S)=0$ by showing that the cokernel of $H^0(V^*) \otimes \OO_S \to V^*$ has no non-zero maps to $\OO_S$, hence any non-zero map $V^* \to \OO_S$ would force $V^*$ to have an $\OO_S$-summand, which is impossible since our extension was chosen as a point of $\op{Gr}\big(r,\op{Ext}(L \otimes \mathcal{I}_W,\OO_S) \big)$.

For (f), we note that the sections $\OO_S^r \to V^*$ satisfy the claim since they drop rank only on $W$. Any other section $s$ of $V^*$ induces a section of $L \otimes \mathcal{I}_W$ corresponding to a curve $C$ of class $L$ containing $W$, and one can use the snake lemma to see that the vanishing locus of $s$ is contained in $C \setminus W$. Thus if $s$ vanished on a curve $C_1$, $C$ would split as $C = C_1 \cup C_2$ with $W \subset C_2$. We can use the $N$-very ample assumption on $L$ to rule this out as follows. A basis $\mathcal{B}_{\text{eff}}$ of the effective cone in $\op{Pic}S$ is given in \cite{batyrev_cox_2004} (Corollary 3.3). A brute force check reveals that
\[
	D.D' \le 3 \quad \text{and} \quad \chi(-D) = \begin{cases} 1 & \text{if $D = -K_{S_8}$} \\ 0 & \text{otherwise} \end{cases} \quad \text{for all $D,D' \in \mathcal{B}_{\text{eff}}$}.
\]
Since $L$ is $N$-very ample, $L.D \ge N$ for all $D \in \mathcal{B}_{\text{eff}}$, so the following lemma guarantees that $h^0(L)-h^0(L_1) > N$, where $L_1 = \OO_S(C_1)$. But then $|W| = h^0(L)-(n-1)(r-1) > h^0(L_1)$ and $W$ is general, so $C_1$ cannot possibly contain $W$, contradiction.
\end{proof}

\begin{lem}\label{drop in sections} Let $S$ be a del Pezzo surface. Let $k \ge 3$, let $L$ be a line bundle such that $L-D \ge 0$ and $L.D \ge k$ for all $D \in \mathcal{B}_{\op{eff}}$, and assume also that $-L.K_{S_8} \ge k+1$ in the case $S=S_8$. Then $h^0(L)-h^0(L-D) > k$ for all $D > 0$.
\end{lem}

\begin{proof} It suffices to prove the inequality for all $D \in \mathcal{B}_{\op{eff}}$. For all $D,D' \in \mathcal{B}_{\op{eff}}$, $(L-D).D' = L.D' - D.D' \ge 3 - 3 = 0$, which implies that $L-D$ is nef and hence has vanishing higher cohomology (\cite{knutsen_exceptional_2003}). Thus
\[
	h^0(L)-h^0(L-D)=\chi(L)-\chi(L-D) = L.D + 1 - \chi(-D),
\]
so the claim follows from the computation of $\chi(-D)$ given above.
\end{proof}

\subsection{Projective bundle}\label{ss: projective bundle} Assume $V^*$ is chosen as in the previous proposition with $(n,r) \notin \{\,(2,2),(2,3),(3,2)\,\}$ to ensure $V^*$ is globally generated. The snake lemma yields a commutative diagram
\[\xymatrix{
	&& G \ar[d] \ar@{=}[r] & G \ar[d] \\
	0 \ar[r] & \OO_S^r \ar@{=}[d] \ar[r] & H^0(V^*)\otimes \OO_S \ar[d] \ar[r] & H^0(L \otimes \mathcal{I}_W) \otimes \OO_S \ar[d] \ar[r] & 0 \\
	0 \ar[r] & \OO_S^r \ar[r] & V^* \ar[r] & L \otimes \mathcal{I}_W \ar[r] & 0
}\]
in which the two vertical sequences defining the kernel $G$ are exact. $G$ is locally free, has no cohomology, and its dual fits in the short exact sequence $0 \to V \to H^0(V^*)^* \otimes \OO_S \to G^* \to 0$. Since $h^0(V)=0$, the diagram yields inclusions $H^0(L \otimes \mathcal{I}_W)^* \hookrightarrow H^0(V^*)^* \hookrightarrow H^0(G^*)$, which make explicit the natural isomorphism
\[
\op{Gr}\big(r,\op{Ext}^1(L \otimes \mathcal{I}_W,\OO_S)\big) \cong \op{Gr}(r,H^0(G^*)/H^0(L \otimes \mathcal{I}_W)^*)
\]
coming from the isomorphism $\op{Ext}^1(L \otimes \mathcal{I}_W,\OO_S) \cong H^0(G^*)/H^0(L \otimes \mathcal{I}_W)^*$ induced by the cohomology long exact sequence associated to the right vertical sequence in the diagram above tensored by $\omega_S$. Thus generic choices of extensions $V^*$ correspond to generic subspaces of $H^0(G^*)$ containing $H^0(L \otimes \mathcal{I}_W)^*$.

In fact, we now show that the image of $H^0(L \otimes \mathcal{I}_W)^* \hookrightarrow H^0(G^*)$ is a general subspace of $H^0(G^*)$ for general $W$ satisfying the genericity condition of \ref{ss:choosingV}, which implies that the subspace $H^0(V^*)^* \hookrightarrow H^0(G^*)$ is general when $V^*$ is general. Since $G^*$ is globally generated, $N = (n-1)(r-1)$ general sections of $G^*$ yield an exact sequence
\[
	0 \to L^* \to \OO_S^N \to G^* \to \OO_W \to 0
\]
whose dual sequence
\[
	0 \to G \to \OO_S^N \to L \otimes \mathcal{I}_W \to 0
\]
shows that the $N$ general sections are dual to the sections of some $L \otimes \mathcal{I}_W$, and $W$ has the right genericity conditions since $G$ has no cohomology.

The fibers of $G$ parametrize the sections of $V^*$ that vanish at points, and we can compile them into a map whose $n$-fold points exactly correspond to sections of $V^*$ vanishing at $n$ points:

\begin{prop}\label{proj bundle}
There is a map $f \colon X = P(G) \to P\big( H^0(V^*)\big) = Y$ described in two ways as
\begin{enumerate}
\item the composition $\xymatrix{P(G) \ar@{^(->}[r]^-{i} & P\big(H^0(V^*)\big) \times S \ar@{->>}[r] & P\big(H^0(V^*)\big)}$ of the projectivization of $G \to H^0(V^*) \otimes \OO_S$ and projection to the first factor;
\item the composition $\xymatrix{ P(G) \ar[r] & \PP H^0(G^*) \ar@{-->}[r] & Y}$ of the map induced by \linebreak[4] $\OO_{P(G)}(1)$ and a general projection onto a projective space of dimension $n(r-1)$;
\end{enumerate}
which has the following properties:
\begin{enumerate}[(a)]
\item $f$ is linear inclusion on fibers of $\pi \colon X \to S$;
\item the image of $f$ spans $Y$;
\item $f^* \big( \OO_{Y}(1) \big) = \OO_X(1)$;
\item For every $s \in H^0(V^*)$, the inverse image $f^{-1}(s)$ viewed in $\{s\} \times S \simeq S$ using $i$ is equal to the scheme-theoretic vanishing locus of $s$ as a section of $V^*$.
\end{enumerate}
\end{prop}

\begin{proof} Part (a) follows from description (1) since both maps preserve fibers of the projective bundles. Part (b) is clear from (2) and can be checked from (1) using $h^0(V)=0$. Part (c) is clear from (2). For (d), note that the fiber of $\pi$ over $p$ coincides with the bundle fiber $G_p$, whose image in $Y$ is the sections of $V^*$ vanishing at $p$. The scheme inverse image $f^{-1}(s)$ identifies all the points $p$ at which $s$ vanishes and supplies the appropriate scheme structure.
\end{proof}

Part (d) of the proposition identifies the sections of $V^*$ vanishing at $n$ points as the $n$-fold points of the map $f$. To count these $n$-fold points, we need an $n$-fold point formula. Unfortunately, since our $n$-fold point loci are zero-dimensional, Kleiman's $n$-fold point formulas (\cite{kleiman_multiple-point_1981}) will only work if $f$ has corank 1, namely if its derivative drops rank by at most 1. By an expected codimension computation (\cite{kazarian_multisingularities_2003}), $f$ should have corank 1 when
\[
	\text{$n=2,3$ and any $r$}; \qquad \text{$n=4$ and $r \le 4$}; \qquad \text{ $n=5,6$ and $r=2$};
\]
and one can check that the resulting computations using Kleiman's $n$-fold point formulas agree with $\chi\left(\OO_{S^{[n]}}(L_n - r\tfrac{B}{2})\right)$. In these cases, we expect but have not been able to prove that Ran's results on general projections (\cite{ran_fibres_2015}, \cite{ran_unobstructedness_2015}) should guarantee that Kleiman's formulas are counting only ordinary multiple points. But in general $f$ has corank 2: although $f$ is a linear inclusion on the projective fibers of $P(G)$, the derivative of $f$ can and will vanish in both directions coming from the base $S$.

Since the algebraic multiple point theory does not seem to cover corank 2 maps, we pass to the topological theory. There we can view (d) as a dictionary between vanishing loci of sections and all multisingularities, which we now explain.

\subsection{Multisingularities}\label{multisingularities}

The following brief introduction to multisingularities is based on the more-detailed discussion in \cite{marangell_general_2010}.

Let $f \colon X \to Y$ be a holomorphic map of complex manifolds of dimensions $m = \dim X$ and $n = \dim Y$ such that $\ell = \dim Y - \dim X \ge 1$. The germ of $f$ at each point $p \in X$ is a map $f_p \colon (\mathbb{C}^{m},0) \to (\mathbb{C}^{n},0)$ defined by $n$ power series $f_1,\dots,f_n \in \mathbb{C}\llbracket x_1,\dots,x_m\rrbracket$ with no constant term. The \emph{local algebra} of $f$ at $p$ is defined to be $Q_{f,p} = \mathbb{C}\llbracket x_1,\dots,x_m \rrbracket/(f_1,\dots,f_n)$ and its isomorphism class characterizes the \emph{contact singularity} of $f$ at $p$. We will use the notation $\alpha$ to denote a general contact singularity and $Q_{\alpha}$ to denote the corresponding isomorphism class of local algebras. We will consider only singularities $\alpha$ for which $Q_{\alpha}$ is finite dimensional over $\mathbb{C}$, which are called \emph{finite} singularities. We say a singularity $\alpha$ has \emph{corank} $r$ if $Q_{\alpha}$ can by minimally generated as an algebra by $r$ generators, which is equivalent to the derivative dropping rank by $r$. The map $f$ has corank $r$ if all singularities of $f$ have corank $\le r$.

The unique corank 0 singularity, denoted $A_0$, has $Q_{A_0}\simeq\mathbb{C}$. $A_0$ singularities are points at which $f$ is an immersion, so $f$ has corank 0 if and only if it is an immersion. The corank 1 singularities, often called \emph{Morin singularities}, are denoted $A_1,A_2,\dots$ and have $Q_{A_i}\simeq\mathbb{C}[t]/(t^{i+1})$. The classification of corank $\ge 2$ singularities becomes complicated.

\begin{ex} The normalization $\op{Spec} \mathbb{C}[t] \to \op{Spec} \mathbb{C}[x,y]/(y^2-x^3)$ defined by $x \mapsto t^2$, $y \mapsto t^3$ of the cuspidal plane cubic curve has an $A_1$ singularity above the singular point since $\mathbb{C}\llbracket t\rrbracket/(t^2,t^3) \simeq \mathbb{C}[t]/(t^2)$.
\end{ex}

The types of singularities give a stratification of $X$ but not a stratification of $Y$ since a point of $Y$ may have preimages with different singularities.

\begin{defn} The map $f \colon X \to Y$ has a \emph{multisingularity} of type $\underline{\alpha}=(\alpha_1,\dots,\alpha_k)$ at a point $p_1 \in X$ if $f$ has singularity $\alpha_1$ at $p_1$ and if the other preimages $p_2,\dots,p_k$ of $f(p_1)$ have singularities $\alpha_2,\dots,\alpha_k$. We use $Q_{\underline{\alpha}}$ to denote the list of local algebras $(Q_{\alpha_1},\dots,Q_{\alpha_k})$ and we define the \emph{length} of $\underline{\alpha}$ to be $\sum_{i=1}^k \dim_{\mathbb{C}} Q_{\alpha_i}$.
\end{defn}
We will be most interested in multisingularities of type $\underline{\alpha} = (A_0,\dots,A_0) = A_0^k$, which we call \emph{$k$-fold points} of $f$ since they correspond to points $q \in Y$ such that the preimage of $q$ under $f$ is $\{ p_1,\dots,p_k \}$ and $f$ is an immersion at each $p_i$.

\begin{ex}\label{node} The normalization $\op{Spec}\mathbb{C}[t] \to \op{Spec}\mathbb{C}[x,y]/\big(y^2-x^2(x+1)\big)$ defined by $x \mapsto t^2-1$, $y \mapsto t(t^2-1)$ of the nodal plane cubic curve has a double point above the node.
\end{ex}

We can stratify $X$ and $Y$ into multisingularity types.
If $\underline{\alpha}=(\alpha_1,\dots,\alpha_k)$ denotes a multisingularity, then the locus
\[
	Y_{\underline{\alpha}}=\left\{q \in Y \colon \parbox{16em}{\centering $q$ has exactly $k$ preimages $p_1,\dots,p_k$ and $f$ has singularity $\alpha_i$ at $p_i$}\right\}
\]
of all points in $Y$ over which $f$ has multisingularity $\underline{\alpha}$ is the image of the locus
\[
	X_{\underline{\alpha}}=\left\{p_1 \in X \colon \parbox{18em}{\centering $f(p_1)$ has exactly $k$ preimages $p_1,\dots,p_k$ and $f$ has singularity $\alpha_i$ at $p_i$}\right\}.
\]
The $Y_{\underline{\alpha}}$ stratify $Y$ and the $X_{\underline{\alpha}}$ are a refinement of the stratification of $X$ into singularity types. We let $x_{\underline{\alpha}} \in H^*(X,\mathbb{C})$ and $y_{\underline{\alpha}} \in H^*(Y,\mathbb{C})$ denote the Poincar\'{e}-dual cohomology classes of the closures of $X_{\underline{\alpha}}$ and $Y_{\underline{\alpha}}$, with multiplicities $\# \operatorname{Aut}(\alpha_2,\dots,\alpha_k)$ and $\# \operatorname{Aut}(\underline{\alpha})$, respectively. The multiplicities are chosen to ensure $f_* x_{\underline{\alpha}}=y_{\underline{\alpha}}$.

We will focus most of our attention on the loci of $k$-fold points, which we abbreviate as $X_k$ and $Y_k$. We use $x_k$ to denote the cohomology class of the closure of $X_k$ with multiplicity $\#\op{Aut}(A_0^{k-1})=(k-1)!$, so $x_k = x_{A_0^k}$, but we will break from convention by writing $y_k$ for the closure of $Y_k$, without any scaling, since the unscaled class will have a direct geometric interpretation. With this normalization, $f_* x_k = k! \, y_k = y_{A_0^k}$.

\begin{rem}\label{admissible} As we will see below, there are formulas for computing the classes $x_{\underline{\alpha}}$ and $y_{\underline{\alpha}}$ for certain multisingularities $\underline{\alpha}$. These formulas are only valid when $f$ is \emph{admissible} (\cite{marangell_general_2010} 2.4) . Roughly speaking, there is an infinite-dimensional classifying space $M$ containing a submanifold $M_{\underline{\alpha}}$ for each multisingularity $\underline{\alpha}$. The codimension $\op{codim}\underline{\alpha}$ of $M_{\underline{\alpha}}$ in $M$ is finite. A map $f \colon X \to Y$ induces a map $k_f \colon Y \to M$ such that the locus $Y_{\underline{\alpha}}$ of points in $Y$ over which $f$ has multisingularity $\underline{\alpha}$ is the preimage of $M_{\underline{\alpha}}$ under $k_f$. We say that $f$ is admissible if $k_f$ is transversal to each $M_{\underline{\alpha}}$. In particular, admissibility implies that each $Y_{\underline{\alpha}}$ occurs in the expected codimension $\op{codim}\underline{\alpha}$, but the converse is false. Since there is no known algebraic formulation of admissibility, it is nearly impossible to check that an algebraic map is admissible, so it is common practice in the literature to assume admissibility when the map is constructed geometrically.
\end{rem}

\begin{ex}\label{codim of multising} The codimension (in the codomain $Y$) of an $A_i$ singularity is $\ell + i(\ell+1)$. The codimension of a multisingularity $\underline{\alpha}=(\alpha_1,\dots,\alpha_k)$ is $\op{codim} \underline{\alpha} = \sum_{i=1}^k \op{codim} \alpha_i$. In particular, one can see that among all multisingularities of length $n$, the multiple point locus $A_0^n$ has the smallest codimension $n\ell$.
\end{ex}

Hidden in the word ``expected" in Theorem A is the assumption that the map $f \colon P(G) \to P(H^0(V^*))$ constructed in \S\ref{ss: projective bundle} is admissible. To make this assumption seem plausible, we note that $V$ is chosen using a general point in a Grassmannian, and $f$ is a general projection from the map determined by $\OO_{P(G)}(1)$. One can show that on $\PP^2$ this map is an embedding (using Proposition \ref{p:resolution}), and there is a general expectation that general projections of smooth projective varieties $X \subset \PP^N$ should have only expected singularities (this is a classical problem; some recent papers in this direction are \cite{ran_unobstructedness_2015}, \cite{ran_fibres_2015}, \cite{gruson_smooth_2013}, and the references listed in those papers). In our setting, such a ``general projection conjecture" could take the form  

\begin{conj}\label{c:admissible} Let $S$ be a smooth projective surface, let $G$ be a very ample vector bundle on $S$, and suppose $f \colon P(G) \to \PP^m$ is a general projection of the embedding defined by $\OO_{P(G)}(1)$. Then $f$ is admissible. 
\end{conj}


Returning to our discussion of multisingularities, we note that when $f \colon X \to Y$ is a map of smooth varieties with finite fibers, a multisingularity of type $\underline{\alpha}$ of $f$ over a point $q \in Y$ can be viewed as a closed subscheme $\op{Spec} Q_{\underline{\alpha}} = \bigsqcup \op{Spec} Q_{\alpha_i} \subset X$ supported on the preimage of $y$, which exactly agrees with the scheme-theoretic fiber $f^{-1}(q)$. Intuitively, non-reducedness at a point $p$ in the fiber corresponds to vanishing of derivatives and higher order derivatives of $f$ at $p$, which is encoded by $Q_{f,p}$.

For the map $f \colon X=P(G) \to P\big(H^0(V^*)\big)=Y$ from Proposition \ref{proj bundle}, we can describe $Y_{\underline{\alpha}}$ as the locus of points of $Y$ at which the fiber of $f$ is isomorphic to $\op{Spec}Q_{\underline{\alpha}}$. By Proposition \ref{proj bundle} (d), these fibers are the vanishing loci of the sections parametrized by $Y$, so $Y_{\underline{\alpha}}$ consists of exactly those sections of $V^*$ whose vanishing locus is isomorphic to $\op{Spec}Q_{\underline{\alpha}}$. In particular, the locus of $k$-fold points $Y_k$ is in bijection with the sections of $V^*$ that vanish at exactly $k$ distinct points.

\begin{prop}\label{bijection mult quot} Let $V^*$ be as in Proposition \ref{gen extensions} and 
\[f \colon X=P(G) \to P\big( H^0(V^*)\big)=Y\]
as in Proposition \ref{proj bundle}, so $\dim X = (n-1)(r-1)$, $\dim Y = n(r-1)$, and $\ell = \op{codim}f = r-1$. Assume $f$ is admissible. Then there is a bijection between the closed points of $\op{Quot}(V,(1,0,-n))$ and the $n$-fold point locus $Y_n$, which is finite. In particular, all closed points of $\op{Quot}(V,(1,0,-n))$ are ideal sheaf quotients of reduced zero-dimensional subschemes.
\end{prop}

Since the expected codimension of the $n$-fold point locus is $n\ell = \dim Y$, the admissibility condition guarantees that $Y_n$ is a finite set. Because of our dictionary between multisingularities and vanishing loci of sections of $V^*$, which in turn correspond to quotients of $V$, it suffices to prove that the only quotients of $V$ with Chern character $(1,0,-n)$ are ideal sheaves $\mathcal{I}_Z$, where $Z$ consists of $n$ distinct points. We will do this by ruling out all other possibilities, which are listed in the following lemma.

\begin{lem}\label{l:quotients} Let $F$ be a coherent sheaf with $\op{ch}(F) = (1,0,-n)$ for $n \in \mathbb{Z}_{\ge 0}$. Then $F$ must be one of the following:
\begin{enumerate}
\item $\mathcal{I}_Z$, where $|Z|=n$ and $Z$ is reduced;
\item $\mathcal{I}_Z$, where $|Z|=n$ but $Z$ is not reduced;
\item An extension $0 \to \OO_{Z'} \to F \to \mathcal{I}_Z \to 0$, where $|Z|>n$ and $|Z'|=|Z|-n$;
\item An extension $0 \to \OO_C(D) \oplus \OO_Z \to F \to \OO_S(-C) \otimes \mathcal{I}_{Z'} \to 0$, where $C$ is a curve, $D$ is a divisor on $C$, and $|Z|=n-|Z'|-\op{deg} D$.
\end{enumerate}
\end{lem}

\begin{proof} Let $T$ be the torsion subsheaf of $F$, which fits in an exact sequence $0 \to T \to F \to Q \to 0$. Since $Q$ is torsion free of rank 1, it must be a line bundle tensored by an ideal sheaf. When $T$ is empty or supported on points, $c_1(Q)=0$, so $Q$ is an ideal sheaf of points, which yields cases (1), (2), and (3). If $T$ is supported on a curve of class $C$, then $c_1(Q) = -C$, so $Q = \OO_S(-C) \otimes I_{Z'}$, yielding case (4).
\end{proof}

\begin{proof}[Proof of Proposition \ref{bijection mult quot}] We will show that the only quotients with Chern character $(1,0,-n)$ that can occur are of type (1) in the previous lemma, which are in bijection with $Y_n$. Quotients of type (4) do not occur since $V^*$ has no sections that vanish on curves. Quotients of type (2) and (3) yield multisingularities of length $\ge n$ that are not $n$-fold points, so they occur in codimension $> \op{dim}Y$, namely not at all.
\end{proof}

\begin{rem} For our application to strange duality, it would be ideal to define a natural scheme structure on the $n$-fold point locus $Y_n$ (which should be reduced when $f$ is admissible) and extend the bijection in Proposition \ref{bijection mult quot} to an isomorphism $Y_n \simeq \op{Quot}(V,(1,0,-n))$ that takes into account non-reduced structure.
\end{rem}

\subsection{General multiple point formulas}

In order to compute the number of $n$-fold points of $f \colon X \to Y$ for $n \le 7$, which by Proposition \ref{bijection mult quot} will count the number of closed points in $\op{Quot}(V,(1,0,-n))$, we will use a formula that computes the Poincar\'{e} dual cohomology class $y_n$ of the $n$-fold point locus $Y_n$ as a polynomial in the Chern classes $c_i$ of the virtual normal sheaf $f^* T_Y / T_X$.

Let $X$ and $Y$ be complex manifolds and let $f \colon X \to Y$ be a holomorphic map of codimension $\ell = \dim Y - \dim X > 0$. In \S\ref{multisingularities}, we describe the locus $X_{\underline{\alpha}}$ of multisingularities of type $\underline{\alpha}$ in $X$ and its image $Y_{\underline{\alpha}}$ in $Y$. We let $x_{\underline{\alpha}} \in H^*(X,\mathbb{C})$ and $y_{\underline{\alpha}} \in H^*(Y,\mathbb{C})$ denote the Poincar\'{e}-dual cohomology classes of the closures of these loci, with multiplicities $\# \operatorname{Aut}(\alpha_2,\dots,\alpha_k)$ and $\# \operatorname{Aut}(\underline{\alpha})$, respectively. The multiplicities ensure that $f_* x_{\underline{\alpha}}=y_{\underline{\alpha}}$.

Kazarian discovered a general form for multisingularity formulas that compute $x_{\underline{\alpha}}$ and $y_{\underline{\alpha}}$ (\cite{kazarian_multisingularities_2003} Theorem 3.2). The key ingredient in these formulas is the residual polynomial $R_{\underline{\alpha}}(\ell)$ of $\underline{\alpha}$, which is a universal polynomial in the Chern classes of the virtual normal sheaf of $f$ that depends only on the codimension $\ell$ of $f$. If $\underline{\alpha} = (\alpha)$ is a monosingularity, then $R_{\underline{\alpha}}(\ell)$ agrees with the Thom polynomial and computes the class $x_{\underline{\alpha}}$ when $f$ is admissible. For general $\underline{\alpha}$ and $f$ admissible, there is an iterative formula
\[
  x_{\underline{\alpha}} = R_{\underline{\alpha}}(\ell) + \sum_{1 \in J \subsetneq \{1,\dots,r\}} R_{\underline{\alpha}_J}(\ell) f^*(y_{\underline{\alpha}_{\overline{J}}}) \; \in H^*(X,\mathbb{C}),
\]
where $\underline{\alpha}_J$ is the sub-tuple of $\underline{\alpha}$ defined by $J$, and $\overline{J}$ is the complement of $J$ in $\{ 1,\dots,r \}$. There are two main obstructions to the implementation of this formula. First, the formula is only valid if $f$ is admissible, which we discussed in Remark \ref{admissible}. Second, very few  residual polynomials are known  (see \cite{marangell_general_2010} for a summary).

In the case of the $k$-fold point multisingularity $\underline{\alpha} = A_0^k$, the $R_{A_0^k}(\ell)$ are known for $k \le 7$ by a result of Marangell and Rim\'{a}nyi:
\begin{thm}[\cite{marangell_general_2010}, Theorem 5.1]\label{residual poly}
 For $i \le 6$, $R_{A_0^{i+1}}(\ell) = (-1)^i \, i! \, R_{A_i}(\ell-1)$.
\end{thm}
Here $R_{A_i}(\ell)$ are the Thom polynomials of the $A_i$ singularities introduced in \ref{multisingularities}, which can be computed by a method of Berczi and Szenes \cite{berczi_thom_2012} that we will briefly describe below. Marangell and Rim\'{a}nyi combine their theorem with Kazarian's formula and a computation of $R_{A_3}(\ell)$ to obtain a general quadruple point formula \cite{marangell_general_2010}. Following their approach, we obtain a formula that generalizes the cohomological versions of the Herbert-Ronga double point formula (Corollary \ref{double point formula}) and Kleiman's triple point formula (Corollary \ref{triple pt on Y}) as well as the general quadruple point formula. The shape of the formula is given by
\begin{prop}\label{multiple point formula} Assume $f \colon X \to Y$ is an admissible map of codimension $\ell$ between complex manifolds. Then for $k \le 7$,
\[
	y_k = \frac{1}{k} \sum_{i=0}^{k-1} (-1)^i f_* \big(R_{A_i}(\ell-1)\big) \, y_{k-1-i} \; \in H^{2k\ell}(Y,\mathbb{C}).
\]
\end{prop} 

\begin{proof}
Setting $y_0=[Y]$ for convenience, Kazarian's formula yields
\[
 x_{A_0^k} = \sum_{i=0}^{k-1} \binom{k-1}{i} R_{A_0^{i+1}}(\ell) \, f^*(y_{A_0^{k-1-i}}).
\]
Pushing forward by $f_*$ and using the projection formula, we get
\[
	y_{A_0^k} = \sum_{i=0}^{k-1} \binom{k-1}{i} f_*(R_{A_0^{i+1}}(\ell)) \, y_{A_0^{k-1-i}}.
\]
Now we use $y_{A_0^k} = k! \, y_k$ and Theorem \ref{residual poly} to deduce the stated formula.
\end{proof}

To use Proposition \ref{multiple point formula}, we still need to compute $R_{A_i}(\ell)$ for $i \le 6$. Berczi and Szenes give a strategy for computing $R_{A_i}(\ell)$ by a complicated iterated residue formula (\cite{berczi_thom_2012} Theorem 7.16) involving an auxiliary polynomial $\widehat{Q_i}(z_1,\dots,z_i)$. They compute $\widehat{Q_1}=\widehat{Q_2}=\widehat{Q_3} = 1$, $\widehat{Q_4}=2z_1 + z_2 - z_4$, $\widehat{Q_5}=(2z_1+z_2-z_5)(2z_1^2+3z_1 z_2 - 2z_1 z_5 + 2z_2 z_3 - z_2 z_4 - z_2 z_5 - z_3 z_4 + z_4 z_5)$ and sketch the computation of $\widehat{Q_6}$. We wrote computer code to compute $\widehat{Q_6}$ and to compute the iterated residues in the special case when $c_{\ell+i}=0$ for $i > 4$ (we will see this vanishing in \S\ref{ss:computation}). With this simplifying assumption, we were able to compute $R_{A_i}(\ell)$ up to $i=6$. For example $R_{A_0}(\ell-1)=1$, $R_{A_1}(\ell-1)=c_\ell$, and
\[
	R_{A_2}(\ell-1)=8c_{\ell-4}c_{\ell+4} + 4c_{\ell-3}c_{\ell+3} + 2c_{\ell-2}c_{\ell+2} + c_{\ell-1}c_{\ell+1} + c_{\ell}^2.
\]
The rest of $R_{A_i}(\ell-1)$, which get complicated quickly, can be found in \cite{johnson_two_2016}.

\begin{rem}\label{r:infeasible} It is unknown whether Theorem \ref{residual poly} holds for $i > 6$. If $R_{A_7}(\ell)$ could be computed, then the multiple point formula (Proposition \ref{multiple point formula}) could be computed for $k=8$ and compared to $\chi(\OO_{S^{[8]}}(L_8-r\tfrac{B}{2}))$. Agreement in these computations would be evidence that Theorem \ref{residual poly} holds for $i=7$. Unfortunately, the computation of $\widehat{Q}_i$ for $i \ge 7$ was infeasible for us since the number of variables involved is proportional to $i^2$, so a more efficient algorithm may be necessary.
\end{rem}

\subsection{Computation of multiple point classes}\label{ss:computation}

We complete the proof of Theorem A in the cases $n,r \ge 2$, $n \le 7$, and $(n,r) \notin \{\, (2,2),(2,3),(3,2) \,\}$ by computing the number of $(1,0,-n)$ quotients of $V$ using multiple point formulas and checking that the result matches the formula for $\chi\big(\OO_{S^{[n]}}(L_n - r\tfrac{B}{2})\big)$ derived from Theorem \ref{thm:chi_gen}.

Here is a summary of the ingredients that we need to compute the number of $(1,0,-n)$ quotients:
\begin{enumerate}
\item a map $f \colon X \to Y$ of codimension $\ell = r-1$ whose locus of $n$-fold points $Y_n$ is in bijection with the sections of $V^*$ vanishing at $n$ points;
\item an iterative formula for the cohomology class $y_n$ of $Y_n$ for $n \le 7$, which counts the $n$-fold points if $f$ is admissible;
\item the residual polynomials $R_{A_i}(\ell-1)$ that appear in the formula, assuming the vanishing $c_{\ell+i}=0$ for $i>4$;
\item the relative Chern classes $c_i$ that appear in the $R_{A_i}(\ell-1)$;
\item identities for computing push forwards of products of the $c_i$.
\end{enumerate}
We have not yet explicitly described the last two items, so we do that now. The Chern classes $c_i = c_i(f^*T_Y/T_X)$ of the virtual normal sheaf of $f$ are
\begin{align*}
  c_i=\binom{r+1}{i}\xi^{i} &+ \left[\binom{r}{i-1}L + \binom{r+1}{i-1}K_S\right] \xi^{i-1} \\
  &+ \left[ \binom{r-1}{i-2}c_2(V^*) + \binom{r}{i-2}LK_S + \binom{r+1}{i-2} 2(K_S^2-6\rho) \right] \xi^{i-2},
\end{align*}
where $\xi$ is the divisor class associated to $\OO_X(1)$, $\rho$ is the pullback of the point class $p$ on $S$ under $X \to S$, and $L$, $K_S$, and $c_2(V^*)$ denote the pullbacks of these classes from $S$. This formula is obtained by an elementary computation on the projective bundle $X = P(G)$ using the fact that $Y$ is a projective space. First, we note that $f^*T_Y = (1+\xi)^{n(r-1)+1}$. Second, the sequences $0 \to T_{X/S} \to T_X \to \pi^* T_S \to 0$ and the relative Euler sequence $0 \to \OO_X \to \pi^*G \otimes \OO_X(1) \to T_{X/S} \to 0$, together with $c(T_S)=1-K_S+(12\chi(\OO_S)p-K_S^2)$ yield $c(T_X) = c(\pi^* G \otimes \OO_X(1))(1 - K_S + (12\rho -K_S^2))$. Third, we compute $c(\pi^* G \otimes \OO_X(1))$ using \cite{fulton_intersection_1998} (3.2.3b). Finally, we extract the degree $i$ part of the appropriate product to get the formula for $c_i$.

The push forward identities are simple since classes are determined by their degree on the projective space $Y$. Let $H$ denote the hyperplane class on $Y$. Let $\delta$ be the divisor class on $X$ obtained by pulling back a divisor class $d$ on $S$. Then the projection formula yields
\[
	f_*(\xi^k) = ([S].c_2(V^*)) H^{r-1+k}; \quad f_*(\delta \xi^k)=(d.L)H^{r+k}; \quad f_*(\rho \xi^k) = H^{r+1+k}.
\]

We wrote computer code in Sage to compute the iterative multiple point formula for $y_n$ for $n \le 7$. Here are the results for $n \le 3$, with the substitution $L^2=2\chi(L)+L.K_S-2$ to reduce the length of the output:
\begin{align*}
	y_1 &= \chi(L); \\
	y_2 &= \tfrac{1}{2}\chi(L)^2 + \chi(L) \big(-
r^2 + \tfrac{1}{2} \big) + K_S^2\big(-\tfrac{r^4}{24}+\tfrac{r^3}{12}+\tfrac{r^2}{24}- \tfrac{r}{12}\big) + L.K_S\big(-\tfrac{r^3}{6}+\tfrac{r}{6} \big) \\
    &\quad+ \tfrac{r^4}{4} - \tfrac{r^2}{4}; \displaybreak[0]\\
	y_3 &= \tfrac{1}{6}\chi(L)^3 + \chi(L)^2 \big(-r^2+\tfrac{1}{2}\big) +\chi(L) \big(\tfrac{7r^4}{4} - \tfrac{7r^2}{4} + \tfrac{1}{3} \big) \\
    &\quad+ \chi(L)K_S^2 \big(- \tfrac{r^4}{24} + \tfrac{r^3}{12} + \tfrac{r^2}{24} - \tfrac{r}{12} \big) 
   + \chi(L)L.K_S \big( - \tfrac{r^3}{6} + \tfrac{r}{6} \big) \\
    &\quad+ K_S^2 \big(\tfrac{97r^6}{720}- \tfrac{17r^5}{80} - \tfrac{31r^4}{144} + \tfrac{5r^3}{16} +
\tfrac{29r^2}{360} - \tfrac{r}{10} \big)
	+ L.K_S \big(\tfrac{17r^5}{40} - \tfrac{5r^3}{8} + \tfrac{r}{5} \big) \\
   &\quad - \tfrac{2r^6}{3} + r^4  - \tfrac{r^2}{3}.
\end{align*}
These formulas for $y_n$ match $\chi\big(\OO_{S^{[n]}}(L_n - r\tfrac{B}{2})\big)$. Our code checks the remaining cases $4 \le n \le 7$ as well. The explicit formulas can be found in \cite{johnson_two_2016}.

\section{Genericity on $\PP^2$}\label{s:genericity}
As evidence that our computations in Theorem A are meaningful, we will now prove Theorem B, which exhibits some vector bundles $V$ on $\PP^2$ that have the right invariants as well as finitely many $(1,0,-n)$ quotients that are all ideal sheaves.

We consider short exact sequences
\[
	0 \to E \to V \to \mathcal{I}_Z \to 0
\]
of sheaves on $\PP^2$, where $Z$ is a zero-dimensional subscheme of $\PP^2$ of length $n$, $e= \op{ch}(E) = (r, -\lambda, (n-1)r - \tfrac{3}{2}\lambda)$, and thus $v=\op{ch}(V) = (r+1, -\lambda, (n-1)r-n - \tfrac{3}{2}\lambda)$. It will often be convenient to consider the dual long exact sequence
\[
	0 \to \OO \to V^* \to E^* \to \OO_Z \to 0
\]
in which the section $\OO \to V^*$ vanishes along $Z$, so the cokernel fails to be locally free along $Z$. We write $e^\vee$ and $v^\vee$ for the dual invariants. We assume $r \ge 2$, $n \ge 1$, and $\lambda$ sufficiently large relative to $r$ and $n$ so that $M(v)$ is positive dimensional of the expected dimension, as guaranteed by

\begin{thm}[\cite{le_potier_lectures_1997}] \label{t:delta_curve}
There exists a positive dimensional moduli space $M(\xi)$ if and only if $\chi(\xi)$ and $c_1(\xi)$ are integral and $\Delta(\xi) \ge \delta(\mu(\xi))$. In this case $M(\xi)$ is a normal, irreducible, factorial projective variety of dimension $1 - \chi(\xi, \xi)$.
\end{thm}

\noindent Here the discriminant $\Delta(\xi)$ is defined by the formula
\[
\Delta(\xi)=\tfrac12 \mu(\xi)^2 - \op{ch}_2(\xi)/r(\xi),
\]
where $r$ is the rank and $\mu(\xi) = \text{c}_1(\xi)/r(\xi)$ is the slope, while the function $\delta$ has a complicated fractal-like structure but is bounded above by 1, so $\Delta(v) \ge 1$ is sufficient.

As we will prove later in Proposition \ref{p:gen_res}, if $\xi$ is a Chern character on $\PP^2$ satisfying certain inequalities, then general stable sheaves $G$ in $M(\xi)$ have resolutions of the form
\[
	0 \to \OO(-2)^{\gamma} \to \OO(-1)^\beta \oplus \OO^\alpha \to G \to 0 \tag{$\dagger$},
\]
where $\alpha,\beta,\gamma \ge 0$ are uniquely determined by $\xi$. These resolutions will play a critical role in our study of general vector bundles, and we will refer to them as ($\dagger$)-resolutions. In particular, applying Proposition $\ref{p:gen_res}$ to the invariants $e^{\vee}$ and $v^{\vee}$ yields
\begin{prop}\label{p:EV_res} Assume $(n-1)r < \tfrac{3-\sqrt{5}}{2} \lambda$. Then a general sheaf $E^*$ in $M(e^\vee)$ has a ($\dagger$)-resolution
\[
	0 \to \OO(-2)^{c} \to \OO(-1)^{b} \oplus \OO^{a} \to E^* \to 0
\]
and a general sheaf $V^*$ in $M(v^{\vee})$ has a ($\dagger$)-resolution
\[
	0 \to \OO(-2)^{c+n} \to \OO(-1)^{b+2n} \oplus \OO^{a+1-n} \to V^* \to 0,
\]
where
\[
	a = nr, \quad b = \lambda - 2(n-1)r, \quad c = \lambda - (n-1)r.
\]
Conversely, cokernels of general maps $\OO(-2)^c \to \OO(-1)^b \oplus \OO^a$ and $\OO(-2)^{c+n} \to \OO(-1)^{b+2n} \oplus \OO^{a+1-n}$ are Gieseker-semistable.
\end{prop}

\begin{rem} Throughout this section, $a,b,c$ will be as in the proposition, and $A = a+1-n$, $B = b+2n$, $C = c+n$ will be used to simplify notation. Since general $E^*$ and $V^*$ in moduli are locally free, we can dualize the ($\dagger$)-resolutions in the proposition to get resolutions $0 \to E \to \OO^a \oplus \OO(1)^b \to \OO(2)^c \to 0$ and $0 \to V \to \OO^A \oplus \OO(1)^B \to \OO(2)^C \to 0$ for general $E \in M(e)$ and $V \in M(v)$.
\end{rem}

The description of general $V^*$ in Proposition \ref{p:EV_res} is consistent with the way we constructed $V^*$ in \S3 and \S4:

\begin{cor}\label{c:same_V*} General sheaves in $M(v^\vee)$ coincide with general extensions
\[
	0 \to \OO^r \to V^* \to \OO(\lambda) \otimes \mathcal{I}_W \to 0,
\]
where $W \subset \PP^2$ is a general zero-dimensional subscheme of length $\binom{\lambda+2}{2} - (n-1)(r-1)$. If $(n-1)(r-1) \ge 3$, then general sheaves in $M(v^\vee)$ are globally generated.\footnote{Proposition \ref{p:gg} is a more general result about global generation of stable sheaves on $\PP^2$.}
\end{cor}

\begin{proof} If $V^*$ has a general ($\dagger$)-resolution, then the cokernel of a general map $\OO^r \to V^*$ is also the cokernel of a general map
\[	
	\OO(-2)^C \to \OO(-1)^B \oplus \OO^{A-r},
\]
so it is torsion-free and thus of the form $\OO(\lambda) \otimes \mathcal{I}_W$. The fact that both $W$ and the extension are general follows from a dimension count showing that the dimension of extensions agrees with $\dim M(v^\vee)$. Since $\OO(\lambda) \otimes \mathcal{I}_W$ is globally generated when it has $\ge 3$ sections, so is $V^*$.
\end{proof}

It will be convenient to have a criterion for detecting when a given coherent sheaf on $\PP^2$ has a ($\dagger$)-resolution. The following proposition applies to coherent sheaves with arbitrary Chern classes that may not be stable or locally-free.

\begin{prop}\label{p:res_crit} A coherent sheaf $G$ has a ($\dagger$)-resolution if and only if $h^0(G(-1))=h^1(G)=h^2(G(-1))=0$, and $\op{Hom}(\OO(-1),\OO) \otimes H^0(G) \to H^0(G(1))$ is injective.
\end{prop}

\begin{proof} The ($\implies$) direction follows from the fact that line bundles on $\PP^2$ have no first cohomology and the observation that $\op{Hom}(\OO(-1),\OO) \otimes H^0(\OO^{\alpha}) \to H^0(\OO(1)^{\alpha})$ is an isomorphism for all $\alpha$.

For ($\impliedby$), set $\alpha = h^0(G)$ and $\beta = h^0(G(1))-3\alpha$. The vanishing $h^1(G)=h^2(G(-1))=0$ implies that $G$ has Castelnuovo-Mumford regularity $\le 1$, so $G(1)$ is globally generated. Starting with the surjection $\OO^{3\alpha+\beta} \twoheadrightarrow G(1)$, observe that $3\alpha$ of these sections factor through $\OO(1)$, and conclude that there is a (non-canonical) surjection $\OO^{\beta} \oplus \OO(1)^{\alpha} \twoheadrightarrow G(1)$, which we twist to get a short exact sequence
\[
	0 \to K \to \OO(-1)^{\beta} \oplus \OO^{\alpha} \xrightarrow{f} G \to 0.
\]
Since line bundles have no first cohomology and $f$ induces an isomorphism on global sections, $h^1(K)=0$ and hence $h^1(K(n)) = 0$ for all $n \ge 0$. By assumption $h^0(G(-1))=0$, so $h^0(G(n))=0$ for all $n < 0$, which implies $h^1(K(n))=0$ for all $n < 0$. Since all twists of $K$ have vanishing first cohomology, $K$ must be a direct sum of line bundles by the splitting criterion of Horrocks (\cite{okonek_vector_1980}). $K$ cannot contain any $\OO$ or $\OO(-1)$-summands by construction, nor can it contain $\OO(-n)$-summands for $n \ge 3$ since $h^1(G)=0$, so $K \simeq \OO(-2)^{\gamma}$ for some $\gamma$.
\end{proof}


\subsection{Set-up}

Instead of working with the moduli space $M(v)$, it will be more convenient to work with the resolution space. We define $R(v)$ to be the open subset of the projective space $\PP^N = P\big(\op{Hom}(\OO^A \oplus \OO(1)^B,\OO(2)^C)\big)$ consisting of surjective morphisms (whose kernels thus have the right invariants). $R(v)$ has dimension $N = 6AC + 3BC - 1$ and the subset of resolutions of stable sheaves is open and dense by Proposition \ref{p:EV_res}. A useful feature of $R(v)$ is that it has a universal family $\mathcal{V}$ over $R(v) \times \PP^2$ defined as the kernel of a morphism of vector bundles
\[
	q^* \OO_{R(v)}(-1) \otimes p^*( \OO_{\PP^2}^A \oplus \OO_{\PP^2}(1)^B) \to p^* \OO_{\PP^2}(2)^C,
\]
where $R(v) \xleftarrow{q} R(v) \times \PP^2 \xrightarrow{p} \PP^2$ are the projections. The morphism is defined by the general matrix of linear and quadratic forms in the coordinates of $\PP^2$, where the projective coordinates of $R(v)$ parametrize the coefficients of the linear and quadratic forms.

Since we are interested in $\mathcal{I}_Z$ quotients of $V$, we consider commutative diagrams
\[\xymatrix{
	\OO^A \oplus \OO(1)^B \ar@{->>}[d]_{\pi} \ar[r]^-f & \OO(2)^C \ar@{->>}[d]^g \\ \OO \ar@{->>}[r]_{\pi_Z} & \OO_Z
} \tag{$\star$} \]
in which $f$ need not be surjective, $g$ is necessarily surjective, $Z$ varies in $(\PP^2)^{[n]}$, $\pi$ is induced by a rank 1 quotient $\OO^A \twoheadrightarrow \OO$, and $\pi_Z$ is the canonical map.

In the case when $f$ is surjective, we claim that these commutative diagrams are in bijection with the maps $V \to \mathcal{I}_Z$. The induced map on kernels of the rows in the diagram is of the form $V \to \mathcal{I}_Z$. Conversely, the map $V \to \OO^A$ can be identified with $V \to \op{Hom}(V,\OO)^* \otimes \OO$, so given any map $V \to \mathcal{I}_Z$ (possibly non-surjective), the composition $V \to \mathcal{I}_Z \to \OO$ factors through $V \to \OO^A$ and yields a diagram ($\star$).

To globalize the above diagrams, consider the vector bundle $\mathcal{E}=q_*(\OO_{\mathcal{Z}} \otimes p^*\OO(-2))$ on $(\PP^2)^{[n]}$, where $\mathcal{Z} \subset (\PP^2)^{[n]} \times \PP^2$ is the universal subscheme and we abuse notation by again writing $q,p$ for the projections. The fiber of $\mathcal{E}$ at $Z$ is $H^0(\OO_Z \otimes p^*\OO(-2)) \cong \op{Hom}(\OO(2),\OO_Z)$ by Cohomology and Base Change (\cite{hartshorne_algebraic_1977}). Consider the incidence variety
\[
	I_{r,\lambda,n} = \big\{\, (f,g,\pi) \mid \text{$g \circ f = \pi_Z \circ \pi$ in $P\big(\op{Hom}(\OO^A \oplus \OO(1)^B,\OO_Z)\big)$} \,\big\}
\]
contained in $\PP^N \times P(\mathcal{E}^C) \times \PP^{A-1}$, which we will usually view as a family over $\PP^N$. We will prove
\begin{prop}\label{p:incidence} Let $\lambda \gg 0$. Then
\begin{enumerate}[(a)]
\item $I_{r,\lambda,n}$ has a unique component of dimension $N$ and any other components have strictly smaller dimension;
\item For $\lambda \gg 0$, the Chern characters $e$ and $(1,0,-n)$ are candidates for strange duality;
\item There is an $f$ such that $V = \ker f$ has an $\mathcal{I}_Z$ quotient that is an isolated point in $\op{Quot}(V,(1,0,-n))$;
\item For an open set $U \subset R(v)$, the restriction $I_{r,\lambda,n}|_U$ coincides with the relative quot scheme $\op{Quot}(\mathcal{V}|_U,(1,0,-n))$, whose fibers over $U$ are finite, reduced, and consist of quotients $V \twoheadrightarrow \mathcal{I}_Z$ for which $Z$ consists of $n$ general distinct points.
\end{enumerate}
\end{prop}


\subsection{Dimension count}
To prove (a) in Proposition \ref{p:incidence}, we stratify the fiber $P(\op{Hom}(\OO(2)^C,\OO_Z))$ of $P(\mathcal{E}^C)$ over $Z \in (\PP^2)^{[n]}$ as a union of varieties over which we can control the dimension of $I_{r,\lambda,n}$. Let
\[
	W_k \subset \op{Hom}(\OO(2)^C,\OO_Z)
\]
be the locally closed subset of maps $g$ that factor through a map $\OO(2)^C \to \OO(2)^k$ but not through a map $\OO(2)^C \to \OO(2)^{k-1}$. Since $\op{hom}(\OO(2),\OO_Z)=n$, we can write
\[
	P(\op{Hom}(\OO(2)^C,\OO_Z)) = P(W_1) \sqcup \cdots \sqcup P(W_n).
\]
The codimension of $W_k$ in $\op{Hom}(\OO(2)^C,\OO_Z)$ is $(n-k)(C-k)$, which is computed by adding the dimension of the Grassmannian $\op{Gr}(C,k)$ and $\op{hom}(\OO(2)^k,\OO_Z)$.

We compute the dimension of $I_{r,\lambda,n}$ over these strata by describing the fibers. For each fixed pair $(g,\pi) \in P(W_k) \times \PP^{A-1}$ over $Z$, there is an exact sequence
\[
	\op{Hom}(\OO^A \oplus \OO(1)^B,\OO(2)^C) \xrightarrow{g_*} \op{Hom}(\OO^A \oplus \OO(1)^B,\OO_Z) \to \op{Ext}^1(\OO^A \oplus \OO(1)^B,\ker g) \to 0,
\]
and the fiber in $I_{r,\lambda,n}$ over $(g,\pi)$ is the projectivization of the preimage under $g_*$ of $\pi_Z \circ \pi$. We observe that $\op{ext}^1(\OO,\ker g) = h^1(\ker g)$ measures the failure of $H^0(\OO(2)^C) \to H^0(\OO_Z)$ to be surjective and $\op{ext}^1(\OO(1),\ker g) = h^1(\ker g(-1))$ measures the failure of the map $H^0(\OO(1)^C) \to H^0(\OO_Z)$ (induced by $g(-1)$) to be surjective.

In the case $k=n$, both of these maps of global sections are surjective. Since $W_n$ has codimension 0, the dimension of $P(W_n) \times \PP^{A-1}$ is $nC + (A-1)$. The codimension of the fibers of $g_*$ in $\op{Hom}(\OO^A \oplus \OO(1)^B,\OO(2)^C)$ is $\hom(\OO^A \oplus \OO(1)^B,\OO_Z) = n(A+B)$, which is equal to $nC + (A-1) + 2n$. Letting $Z$ vary in $(\PP^2)^{[n]}$, we see that the dimension of $I_{r,\lambda,n}$ is equal to $N = \dim R(v)$ since the following lemma guarantees that the jump in dimension of the fibers of $I_{r,\lambda,n}$ over $W_k$ for $k < n$ is strictly less than the codimension of $W_k$ in $\op{Hom}(\OO(2)^C,\OO_Z)$.

In fact, this dimension count will complete the proof of (a) as follows. For every $N$-dimensional component of $I_{r,\lambda,n}$, the fiber over a general point in $P(\mathcal{E}^C)$ must contain a non-empty open set in the fiber of $I_{r,\lambda,n}$ for dimension reasons. But the fibers of $I_{r,\lambda,n}$ are projective spaces, so the intersection of two open sets contains a non-empty open set, hence there can be only one component.

\begin{lem} Assume $\lambda \gg 0$ and $1 \le k < n$. Then for all $g \in W_k$,
\[
	\op{ext}^1(\OO^A \oplus \OO(1)^B,\op{ker}g)  < \op{codim} W_k = (n-k)(C-k).
\]
\end{lem}

\begin{proof} We think of $H^0(\OO(1))$ as the space of linear forms on $\PP^2$. Fixing $\ell \in H^0(\OO(1))$ that does not vanish on any points in the support of $Z$, we can identify every map $\OO(2) \to \OO_Z$ with a global section of $\OO_Z$ by multiplying by $\ell^2$. Thus $g \in W_k$ determines a $k$-plane $H_{\ell} \subset H^0(\OO_Z)$, and the image of the global section map $H^0(\OO(1)^C) \to H^0(\OO_Z)$ is obtained by multiplying $H_{\ell}$ by the space of rational functions $H^0(\OO(1))/\ell$ and hence contains $H_{\ell}$.

We claim that the rank of $H^0(\OO(1)^C) \to H^0(\OO_Z)$ is $\ge k+1$. If not, then the image must be exactly $H_{\ell}$. But then the image of every map $H^0(\OO(d)^C) \to H^0(\OO_Z)$ obtained by twisting $g$ is also $H_{\ell}$, which contradicts the fact that $H^0(\OO(d)^C) \to H^0(\OO_Z)$ is surjective for $d \gg 0$ by Serre vanishing applied to $\ker g$. By the same argument, the rank of $H^0(\OO(2)^C) \to H^0(\OO_Z)$ is $\ge k+1$ as well. 


Thus the left side of the proposed inequality is
\[
	A \op{ext}^1(\OO,\op{ker}g) + B \op{ext}^1(\OO(1),\op{ker}g) \le (n-k-1)(A+B),
\]
so since $A+B-C = r+1$, it suffices to show $(n-k)(1+r+k) < A+B$, which is achieved by choosing $\lambda$ sufficiently large. 
\end{proof}


\subsection{The Chern characters $e$ and $(1,0,-n)$ are candidates for strange duality}

The non-emptiness of $M(e^{\vee})$ follows from Theorem \ref{t:delta_curve}, $h^2(\hat{E} \otimes \mathcal{I}_Z) = 0$ for all $\hat{E} \in M(e^{\vee})$ and $\mathcal{I}_Z \in (\PP^2)^{[n]}$ by stability, and $\TOR^1(\hat{E},\mathcal{I}_Z)=\TOR^2(\hat{E},\mathcal{I}_Z)=0$ whenever $\hat{E}$ is locally free along $Z$, which holds away from codimension $(r-1)+2 > 2$ in $M(e^{\vee}) \times (\PP^2)^{[n]}$. The last condition to check is that $h^0(\hat{E} \otimes \mathcal{I}_Z) = 0$ for some pair $(\hat{E},\mathcal{I}_Z) \in M(e^{\vee}) \times (\PP^2)^{[n]}$.

\begin{lem}\label{l:theta} Suppose $\lambda \gg 0$. Then general $\hat{E} \in M(e^{\vee})$ and $\mathcal{I}_Z \in (\PP^2)^{[n]}$ satisfy $h^0(\hat{E}\otimes\mathcal{I}_Z)=0$.
\end{lem}

\begin{proof}[Proof of Lemma \ref{l:theta}] Induction on $\lambda$. We will emphasize the dependence of $e$ on $\lambda$ by writing $e = e_{\lambda}$. For each $\lambda$, it suffices to construct a single pair $(\hat{E},\mathcal{I}_Z)$ such that $h^0(\hat{E} \otimes \mathcal{I}_Z)=0$, where $\hat{E}$ is in $M(e_{\lambda}^{\vee})$ or in $R(e_{\lambda}^{\vee})$ (the space of ($\dagger$)-resolutions). This is because the vanishing $h^0(\hat{E} \otimes I_Z)=0$ is open in families by Cohomology and Base Change (\cite{hartshorne_algebraic_1977}).

For a base case, let $\lambda = rk$ for $k$ minimal such that $\chi(\OO_{\PP^2}(k)) \ge n$. Choose $Z$ general of length $n$ and $Z'$ such that $|Z'|+|Z|=\chi(\OO(k))$ and $Z' \cup Z$ is not contained on a curve of degree $k$. Then $\hat{E} = \mathcal{I}_{Z'}(k)^r$ is in $M(e_{\lambda}^{\vee})$ and $h^0(\hat{E} \otimes \mathcal{I}_Z)=0$ since $h^0(\mathcal{I}_{Z' \cup Z}(k))=0$ by our choice of $Z'$.

For the inductive step $\lambda \implies \lambda + 1$, we may assume there exist general $E_\lambda^* \in R(e_{\lambda}^{\vee})$ and $\mathcal{I}_Z \in (\PP^2)^{[n]}$ satisfying $h^0(E_\lambda^* \otimes \mathcal{I}_Z)=0$. Since $E_\lambda^*$ is general, it is locally free with dual $E_\lambda$. Let $\ell$ be a general line in $\PP^2$, for which $E_{\lambda}|_{\ell}$ splits into a direct sum of line bundles, at least two of which have negative degree by the Grauert-M\"{u}lich Theorem (\cite{huybrechts_geometry_2010}). Choosing a surjection $E_{\lambda}|_{\ell} \twoheadrightarrow \OO_{\ell}(2)$ yields an elementary modification
\[
	0 \to E_{\lambda+1} \to E_{\lambda} \to \OO_{\ell}(2) \to 0
\]
where restricting the sequence to $\ell$ produces the kernel $0 \to \OO_{\ell}(1) \to E_{\lambda+1}|_{\ell} \to E_{\lambda}|_{\ell}$. Exactness in the middle of the sequence
\[
	0 = \op{Hom}(E_{\lambda},\mathcal{I}_Z) \to \op{Hom}(E_{\lambda+1},\mathcal{I}_Z) \to \op{Ext}^1(\OO_{\ell}(2),\mathcal{I}_Z) \simeq H^1(\OO_{\ell}(-1)) = 0 
\]
yields the vanishing $h^0(E_{\lambda+1}^* \otimes \mathcal{I}_Z)=\op{hom}(E_{\lambda+1},\mathcal{I}_Z)=0$.

Consider the dual sequence of sheaves
\[
	0 \to E_{\lambda}^* \to E_{\lambda+1}^* \to \OO_{\ell}(-1) \to 0.
\]
Since $E_{\lambda}^*$ has a ($\dagger$)-resolution, $\OO_{\ell}(-1)$ has no cohomology, and $\OO_{\ell}(-2)$ has only cohomology in degree 1, Proposition \ref{p:res_crit} ensures that $E_{\lambda+1}^*$ has a ($\dagger$)-resolution. This completes the proof.
\end{proof}


\subsection{Construction of a suitable $V$ with an isolated quotient}

We prove (c) in Proposition \ref{p:incidence} by constructing a vector bundle $V^*$ with a resolution and an appropriate section, and then dualizing.

Let $E$ be a general vector bundle in $M(e)$ and $\mathcal{I}_Z$ in $(\PP^2)^{[n]}$ be general, which ensures that $\op{hom}(E,\mathcal{I}_Z)=0$ by Lemma \ref{l:theta}. Choose a general surjection $E^* \to \OO_Z$ that induces a surjection on global sections. Let $J$ be the kernel, which fails to be locally free along $Z$. We will show that a general extension $0 \to \OO \to V^* \to J \to 0$ produces $V^*$ that is locally free and has a ($\ddagger$)-resolution. This will complete the argument because dualizing the sequence defining $V^*$ yields the short exact sequence $0 \to E \to V \to \mathcal{I}_Z \to 0$, which is an isolated point of $\op{Quot}(V,(1,0,-n))$ by construction, and $V$ has a resolution since $V^*$ does.

\begin{lem} \begin{enumerate}[(1)]
\item $V^*$ is locally free.
\item $J$ has a ($\dagger$)-resolution.
\item $V^*$ has a ($\dagger$)- resolution.
\end{enumerate}
\end{lem}

\begin{proof}
(1): If an extension $\hat{V}$ fails to be locally free along a subscheme $W \subset Z$, whose residual we denote $W' \subset Z$, then the inclusion $\hat{V} \to \hat{V}^{\vee \vee}$ yields a diagram
\[\xymatrix{
	0 \ar[r] & \OO \ar@{=}[d] \ar[r] & \hat{V} \ar[d] \ar[r] & J \ar[d] \ar[r] & 0 \\
    0 \ar[r] & \OO \ar[r] & \hat{V}^{\vee \vee} \ar[r] & \tilde{J} \ar[r] & 0
}\]
in which the extension in the top row is pulled back from the extension in the bottom row. Here $\tilde{J}$ is a subsheaf of $E^*$ with cokernel $\OO_{W'}$. Since
\[
	\op{ext}^1(\tilde{J},\OO) = h^1(\tilde{J}(-3)) = h^1(E^*(-3))+|W'| < h^1(E)+n = \op{ext}^1(J,\OO)
\]
and there are only finitely many such $\tilde{J}$ (corresponding to finitely many subschemes $W'$ of $Z$), the total locus in $\op{Ext}^1(J,\OO)$ of all extensions pulled back from a $\tilde{J}$ has codimension $\ge 1$. Avoiding this locus yields $V^*$ locally free.

(2): We will use the fact that $E^*$ has a ($\dagger$)-resolution (since it is general in $M(e^\vee)$) and the criterion provided in Proposition \ref{p:res_crit}. Consider the sequence
\[
	0 \to J \to E^* \to \OO_Z \to 0
\]
defining $J$. Since $H^0(E^*) \to H^0(\OO_Z)$ is surjective by construction, $h^1(J)=0$. The other properties for $J$ required in Proposition \ref{p:res_crit} follow immediately from the same properties for $E^*$.

(3): Pulling back the extension $0 \to \OO \to V^* \to J \to 0$ using the map to $J$ in the ($\dagger$)-resolution of $J$ yields a $(\dagger)$-resolution of $V^*$ that looks like the resolution for $J$ with one additional $\OO$.
\end{proof}







\subsection{Deforming the isolated quotient}

To prove (d), we will deform the isolated quotient
\[
	0 \to E \to V \to \mathcal{I}_Z \to 0
\]
constructed above to conclude that general resolutions in $R(v)$ have isolated $\mathcal{I}_Z$ quotients.
Let $\mathcal{V}$ be the universal bundle over $R(v)$ and consider the relative quot scheme $\mathcal{Q} = \op{Quot}(\mathcal{V},(1,0,-n))$ over $R(v)$, which has $[V \to \mathcal{I}_Z]$ as a point in the fiber over $V$. A standard theorem for quot schemes ensures that the dimension of the component of the relative quot scheme containing $[V \to \mathcal{I}_Z]$ is at least
\[
	\op{hom}(E,\mathcal{I}_Z) - \op{ext}^1(E,\mathcal{I}_Z) + \dim R(v) = \dim R(v)
\]
(see \cite{kollar_rational_1996} Theorem 2.15 for a statement in the context of Hilbert schemes). Since the fiber dimension at $[V \to \mathcal{I}_Z]$ is 0, upper semicontinuity implies that the fiber dimension is generically 0, so there is an open set $\mathcal{U} \subset \mathcal{Q}$ consisting entirely of points at which the relative Zariski tangent space is zero-dimensional, namely the map to $R(v)$ is \'{e}tale.

We now show that $\mathcal{U}$ consists entirely of $\mathcal{I}_Z$ quotients by ruling out the other possible sheaves with Chern character $(1,0,-n)$ listed in Lemma \ref{l:quotients}. If the cokernel $F$ in  $0 \to E \to V \to F \to 0$ has zero-dimensional torsion, then $\op{hom}(E,F) > 0$. If $F$ has torsion along a curve, then $V$ has a map to a line bundle of negative degree, which is impossible since $V$ has a resolution. Thus the quotients in $\mathcal{U}$ are all ideal sheaves.

But since every map from sheaves $V$ with resolutions to ideal sheaves $\mathcal{I}_Z$ occurs in $I_{r,\lambda,n}$, this yields an injective morphism $\mathcal{U} \hookrightarrow I_{r,\lambda,n}$. Set theoretically, the map is obtained by extending a quotient $V \twoheadrightarrow \mathcal{I}_Z$ to a diagram ($\star$). We check below that this map is algebraic. The image has dimension $\dim R(v)$, hence must be contained in the unique component of $I_{r,\lambda,n}$ of this dimension. The complement $I_{r,\lambda,n} \setminus \op{im}\mathcal{U}$ of this image is thus of dimension $< \dim R(v)$. Then $U \subset R(v)$, the complement of the image of $I_{r,\lambda,n} \setminus \op{im}\mathcal{U} \to R(v)$, is open in $R(v)$. By construction, the fibers of $I_{r,\lambda,n}$ over $U$ are fully contained in the image of $\mathcal{U}$.

We claim that the composition of the inclusions $I_{r,\lambda,n}|_U \hookrightarrow \mathcal{U} \hookrightarrow \mathcal{Q}|_U$ is an isomorphism. We need to rule out quotients in $\mathcal{Q}|_U$ at which the map to $R(v)$ is not \'{e}tale, which we can do by showing that every such quotient yields a (possibly non-surjective) map to $V \to \mathcal{I}_Z$, which is impossible since the full fibers of $I_{r,\lambda,n}$ are contained in $\mathcal{U}$. A non-etale point in $\mathcal{Q}$ must be either a quotient $0 \to E \to V \to \mathcal{I}_Z \to 0$ for which $\op{hom}(E,\mathcal{I}_Z) > 0$ or a quotient $V \twoheadrightarrow F$ where $F$ is of type (3) or (4) in Lemma \ref{l:quotients}. We have already ruled out type (4) quotients since $V$ has a resolution. The type (3) quotients yield quotients $V \twoheadrightarrow F \twoheadrightarrow \mathcal{I}_W$, where $|W| > n$, and every choice of length-$n$ subscheme $Z \subset W$ gives rise to an inclusion $\mathcal{I}_W \hookrightarrow \mathcal{I}_Z$ and hence a non-surjective map $V \to \mathcal{I}_Z$. Thus the map $I_{r,\lambda,n}|_U \to \mathcal{Q}|_U$ is surjective, hence an isomorphism.

Since $I_{r,\lambda,n}|_U = \mathcal{Q}|_U \to U$ is \'{e}tale, the fibers are finite and reduced. To ensure that all $Z$ occurring as quotients $V \twoheadrightarrow \mathcal{I}_Z$ consist of $n$ general distinct points, we can restrict $I_{r,\lambda,n}$ to any special locus in $(\PP^2)^{[n]}$, take the image in $R(v)$, and shrink $U$ further to avoid this image. 

To complete the proof of Proposition \ref{p:incidence}, we construct the morphism $\mathcal{U} \hookrightarrow I_{r,\lambda,n}$ algebraically by compiling the diagrams ($\star$) into a family. Over $\mathcal{U} \times \PP^2$, the universal quotient $\mathcal{V} \to \mathcal{F}$, where $\mathcal{V}$ is pulled back from $R(v)$, is a family of $(1,0,-n)$ sheaves. Thus there is a map $\phi \colon \mathcal{U} \to (\PP^2)^{[n]}$ such that $\mathcal{F}$, up to a twist by a line bundle $L$ on $\mathcal{U}$, is pulled back from $\mathcal{I}_{\mathcal{Z}}$. We will construct a surjective morphism $\alpha$ yielding a diagram
\[\xymatrix{
	\mathcal{V} \ar[d] \ar[r] & \OO_{R(v)}(-1) \otimes p^*(\OO_{\PP^2}^A \oplus \OO_{\PP^2}(1)^B) \ar[d]^{\alpha} \ar[r] & p^*\OO_{\PP^2}(2)^C \ar[d]^{\beta} \\
	L \otimes (\phi \times \op{id})^*\mathcal{I}_{\mathcal{Z}} \ar[r] & L \otimes \OO_{\mathcal{U} \times \PP^2} \ar[r] & L \otimes (\phi \times \op{id})^* \OO_{\mathcal{Z}}
}\]
on $\mathcal{U} \times \PP^2$, where we have omitted some of the pull backs from the notation. We construct $\alpha$ locally. First, trivialize $\OO_{R(v)}(-1)$ and $L$ on open sets $\mathcal{U}_{i} \subset \mathcal{U}$ such that $\OO_{\mathcal{U}_i}$ has no higher cohomology. On $\mathcal{U}_i \times \PP^2$ the left side of the diagram is
\[\xymatrix{
	\mathcal{V}|_{\mathcal{U}_i} \ar[d] \ar[r] & p^*\OO_{\PP^2}^A \oplus p^*\OO_{\PP^2}(1)^B \ar[d]^{\alpha_i}\\	(\phi|_{\mathcal{U}_i}\times \op{id})^*\mathcal{I}_{\mathcal{Z}} \ar[r] & \OO_{\mathcal{U}_i \times \PP^2}
}\]
and $\alpha_i$ is the unique lift of the composition $[\mathcal{V}|_{\mathcal{U}_i} \to \OO_{\mathcal{U}_i \times \PP^2}]$ in the sequence
\[
	0=\op{Hom}(p^*\OO(2)^C,\OO) \to \op{Hom}(p^*(\OO^A \oplus \OO(1)^B),\OO) \to \op{Hom}(\mathcal{V},\OO) \to \op{Ext}^1(p^*\OO(2)^C,\OO)=0.
\]
The vanishing $\op{ext}^1(p^*\OO_{\PP^2}(2),\OO) = h^1(p^* \OO_{\PP^2}(-2))=0$ follows from the K\"{u}nneth formula and the vanishing of the higher cohomology of $\OO_{\mathcal{U}_i}$. These $\alpha_i$ are surjective since they are nonzero and they glue together into a map $\alpha$ since they agree on fibers over points in $\mathcal{U}$. We get $\beta$ as the induced map on cokernels.

Since $\alpha$ vanishes on $p^*\OO_{\mathbb{P}^2}(1)^B$, it induces a map $\mathcal{U} \to \PP^{A-1}$. To lift the map $\mathcal{U} \xrightarrow{\phi} (\PP^2)^{[n]}$ to $P(\mathcal{E}^{C})$, we push forward $\beta$ by $q$. Taking the product of these maps with the projection $\mathcal{U} \to R(v)$ yields a map $\mathcal{U} \to R(v) \times P(\mathcal{E}^C) \times \PP^{A-1}$ that coincides with our set-theoretic description on closed points. Since $\mathcal{U}$ is smooth, this guarantees that the image is contained in $I_{r,\lambda,n}$, which completes the argument.


\subsection{Resolutions and global generation}

We let $\xi$ be a Chern character on $\PP^2$ satisfying some mild inequalities and prove that general stable sheaves in $M(\xi)$ have particularly nice resolutions. As a corollary, we deduce a statement about when general stable sheaves are globally generated. We ignore the assumptions on $r,\lambda,a,b,c$ made at the beginning of the section.

\begin{prop}\label{p:gen_res} Let $\xi = (r,\lambda,d)$ be a Chern character on $\PP^2$ such that
\[
	r \ge 1, \quad \lambda \ge 0, \quad \chi(\xi) \ge 0, \quad \text{and} \quad d < -\tfrac{\sqrt{5}}{2}\lambda.
\]
Then the general sheaf $G$ in $M(\xi)$ has a resolution of the form
\[
	0 \to \OO(-2)^{c} \to \OO(-1)^{b} \oplus \OO^{a} \to G \to 0,
\]
where
\[
	a = \chi(\xi) = r+\tfrac{3}{2}\lambda+d; \qquad b= -2(\lambda+d); \qquad c = -\tfrac{1}{2}\lambda - d.
\]
Conversely, cokernels of general maps $\OO(-2)^{c} \to \OO(-1)^{b} \oplus \OO^{a}$ are stable sheaves in $M(\xi)$.
\end{prop}

With this result, we can show that a general stable sheaf $G$ of positive rank on $\PP^2$ is globally generated when $\chi(G) \ge \op{rk}(G)+2$. The rank 1 case is an analysis of line bundles and ideal sheaves, and the rank 2 case was known to Le Potier (\cite{le_potier_stabilite_1980}).

\begin{prop}\label{p:gg} Let $\xi=(r,\lambda,d)$ be a Chern character such that $r \ge 1$, $\lambda \ge 0$, and $\chi(\xi) \ge r+2$. Then general sheaves in $M(\xi)$ are globally generated.
\end{prop}

Proposition \ref{p:gen_res} follows from a more general result about resolutions of general sheaves in $M(\xi)$ by triads of exceptional vector bundles. We begin by recalling some basic facts, following \cite{coskun_effective_2014}. A stable vector bundle $E$ on $\PP^2$ is an \emph{exceptional bundle} if $\op{Ext}^1(E,E)=0$, in which case we call the slope $\alpha$ of $E$ an \emph{exceptional slope} and  write $E = E_\alpha$ since $E$ is the unique exceptional bundle with slope $\alpha$. All integers are exceptional slopes since the line bundles $\OO(n)$ are exceptional bundles.

Let $r_\alpha$ be the slope of $E_\alpha$, let $\Delta_\alpha = \tfrac{1}{2}(1-\tfrac{1}{r_\alpha^2})$ be the discriminant of $E_\alpha$, and let $\xi_\alpha = \op{ch}(E_\alpha)$. The set of exceptional slopes $\mathcal{E}$ is in bijection with the dyadic integers via a function $\varepsilon \colon \mathbb{Z}[\tfrac{1}{2}] \to \mathcal{E}$ defined inductively by $\varepsilon(n)=n$ for $n \in \mathbb{Z}$ and by setting
\[
	\varepsilon\left( \tfrac{2p+1}{2^{q+1}}\right) = \varepsilon\left( \tfrac{p}{2^q}\right). \varepsilon\left( \tfrac{p+1}{2^q}\right),
\]
where the product operation on exceptional slopes is defined by $\alpha.\beta = \tfrac{\alpha+\beta}{2} + \tfrac{\Delta_\beta - \Delta_\alpha}{3+\alpha-\beta}$. Since each dyadic integer can be written uniquely as $\tfrac{2p+1}{2^{q+1}}$, the $p$ and $q$ in the equation are uniquely determined, and we call the equation the \emph{standard decomposition} of the exceptional slope $\varepsilon\left( \tfrac{2p+1}{2^{q+1}}\right)$.

To find the right triad for resolving general sheaves in $M(\xi)$, one needs the \emph{corresponding exceptional slope} $\gamma$ of $\xi$. This is obtained by computing
\[
  \mu_0 = -\tfrac{3}{2}-\mu + \sqrt{\tfrac{5}{4}+\mu^2-\tfrac{2d}{r}},
\]
where $\mu = \lambda/r$ is the slope of $\xi$, and then $\gamma$ is the unique exceptional slope satisfying $|\mu_0-\gamma| < x_{\gamma}$, where $x_{\gamma} = \tfrac{3}{2}-\sqrt{\tfrac{9}{4}-\tfrac{1}{r_\gamma^2}}$. Then the resolution is described by
\begin{prop}[\cite{coskun_effective_2014}]\label{p:resolution} Let $\xi$ be a Chern character, let $\gamma$ be the corresponding exceptional slope to $\xi$, and let $\gamma = \alpha.\beta$ be the standard decomposition of $\gamma$. Then:
\begin{enumerate}
\item If $\chi(\xi \otimes \xi_\gamma) \ge 0$, then the general $G \in M(\xi)$ has a resolution
\[
	0 \to E_{-\alpha-3}^{m_1} \to E_{-\beta}^{m_2} \oplus E_{-\gamma}^{m_3} \to G \to 0,
\]
where $m_1 = -\chi(G \otimes E_\alpha)$, $m_2 = -\chi(G \otimes E_{\alpha.\gamma})$, $m_3 = \chi(G \otimes E_\gamma)$.
\item If $\chi(\xi\otimes\xi_\gamma) \le 0$, then the general $G \in M(\xi)$ has a resolution
\[
	0 \to E_{-\alpha-3}^{m_1} \oplus E_{-\gamma-3}^{m_3} \to E_{-\beta}^{m_2} \to G \to 0,
\]
where $m_1 = \chi(G \otimes E_{\gamma.\beta})$, $m_2=\chi(G \otimes E_\beta)$, $m_3=-\chi(G \otimes E_\gamma)$.
\end{enumerate} 
\end{prop}

Our proposition follows from the special case when $\chi(\xi) \ge 0$ and $\gamma = 0$.

\begin{proof}[Proof of Proposition \ref{p:gen_res}] We first claim that $\gamma$, the corresponding exceptional slope to $\xi$, is 0. For $\lambda \ge 0$, the inequality $|\mu_0| < x_0 = \tfrac{3-\sqrt{5}}{2}$ is equivalent to the pair of inequalities
\[
	-(3-\tfrac{\sqrt{5}}{2})\lambda - \tfrac{9-3\sqrt{5}}{2}r < d < -\tfrac{\sqrt{5}}{2}\lambda.
\]
The left inequality is guaranteed by $\chi(\xi) \ge 0$, which yields $d \ge -\tfrac{3}{2}\lambda-r$, while the right inequality is a hypothesis.

Since $\chi(\xi) \ge 0$, the resolution of general sheaves $G$ in $M(\xi)$ is thus of the form
\[
 0 \to \OO(-2)^{c} \to \OO(-1)^{b} \oplus \OO^{a} \to G \to 0.
\]
Applying $\chi$, we see that $a = \chi(\xi) = r+\tfrac{3}{2}\lambda+d$. Additivity on Chern characters yields the relations $a+b-c = r$, $-b+2c = \lambda$, and $\tfrac{1}{2}b - 2c = d$, which can be solved to get $b = -2(\lambda+d)$, $c = -\tfrac{1}{2}\lambda - d$. This proves the first part of the proposition.

The converse follows from a dimension count. The dimension of such resolutions is the dimension of $\op{Hom}(\OO(-2)^{c},\OO(-1)^{b} \oplus \OO^{a})$, minus the dimensions of automorphisms of $\OO(-2)^{c}$ and $\OO(-1)^{b} \oplus \OO^{a}$, plus 1 because we are accounting for scalars twice. The result is exactly
\[
 3bc + 6ac - c^2 - b^2 - a^2 - 3ab + 1 = \lambda^2 - 2rd-r^2+1 = \op{dim} M(\xi),
\]
so the general resolution produces the general vector bundle in $M(\xi)$.
\end{proof}

\begin{proof}[Proof of Proposition \ref{p:gg}] The hardest case is when $\gamma$, the corresponding exceptional slope to $\xi$, is 0. Let $G$ be a general sheaf in $M(\xi)$, which has a resolution given by Proposition \ref{p:gen_res}. The snake lemma applied to the commutative diagram
\[\xymatrix{
  && \OO^{\oplus a} \ar[r]^-{\simeq} \ar[d] & H^0(G) \otimes \OO \ar[d] \\
 0 \ar[r] & \OO(-2)^{\oplus c} \ar@{=}[d] \ar[r] & \OO(-1)^{\oplus b} \oplus \OO^{\oplus a} \ar[r] \ar@{->>}[d] & G \ar[r] \ar@{->>}[d] & 0 \\
  & \OO(-2)^{\oplus c} \ar[r] & \OO(-1)^{\oplus b} \ar[r] & Q
}\]
yields an exact sequence $\OO(-2)^{\oplus c} \to \OO(-1)^{\oplus b} \to Q \to 0$ in which the first map is general if $G$ is general. The assumption that $\chi(\xi) \ge r+2$ is equivalent to $c \ge b+2$, which guarantees that general maps $\OO(-2)^{\oplus c} \to \OO(-1)^{\oplus b}$ are surjective. Thus $Q = 0$, so $G$ is globally generated.

Next, we rule out the case $\gamma > 0$. As in the proof of the Proposition \ref{p:gen_res}, our assumption on $\chi(\xi)$ (even $\chi(\xi) \ge 0$ suffices) ensures that $\mu_0 < \tfrac{3-\sqrt{5}}{2}$. Thus $\gamma \le 0$.

The last case is $\gamma < 0$. In this case $0 \le -\beta < -\gamma$, so the resolution for general $G$ given by Proposition \ref{p:resolution} expresses $G$ as a quotient of exceptional bundles with non-negative slope. The following lemma guarantees that such exceptional bundles are globally generated, hence so is $G$.
\end{proof}

\begin{lem} Let $\gamma \ge 0$ be an exceptional slope. Then $E_{\gamma}$ is globally generated.
\end{lem}

\begin{proof} If $\gamma = n \ge 0$ is an integer, then $E_{\gamma}=\OO_{\PP^2}(n)$ is globally generated. For $\gamma > 0$ a non-integer, we use induction on $q$, where $p \ge 0$ and $q \ge 0$ are the unique integers such that $\gamma = \epsilon(\tfrac{2p+1}{2^{q+1}})$. Let $\gamma = \alpha.\beta$ be the standard decomposition of $\gamma$. By the inductive assumption, $E_\alpha$ is globally generated. Since we can write $\alpha=\epsilon(\tfrac{2p}{2^{q+1}})$, Theorem 2 of \cite{drezet_fibres_1986} implies that $E_{\alpha} \otimes \op{Hom}(E_{\alpha},E_{\gamma}) \to E_{\gamma}$ is surjective. Since $E_{\alpha}$ is globally generated, so is $E_{\gamma}$.
\end{proof}

\section{Concluding Remarks}\label{s:concluding}
Our work in this paper suggests many questions and future directions. First, we ask how generally Theorem B holds. Can it be extended to other del Pezzo surfaces and quotients other than ideal sheaves of points? Second, although the multiple point computations provide an interesting link between quot schemes and singularity theory, admissibility of the map (see Remark \ref{admissible}) is so difficult to check that even strong algebraic genericity properties do not seem to help (see Conjecture \ref{c:admissible}). In the absence of developments in algebraic multisingularity theory, we need a more rigorous way to compute the expected cardinality of finite Quot schemes.

Despite these difficulties with multiple point formulas, it would be interesting to extend the multiple point computation to $n=8$ and compare it with the Euler characteristic of Theorem \ref{thm:chi_gen}. This would provide evidence (or counter-evidence) for extending Theorem \ref{residual poly} to $i=7$. Unfortunately, the current method of computing the auxiliary polynomial $\widehat Q$ requires processing an ideal with too many variables and generators to be computationally feasible (see Remark \ref{r:infeasible}). Generalizing in another direction, the set-up for our multiple point computations could likely be imitated on $\PP^3$ or another Fano threefold. One would then have to investigate whether the results have any enumerative significance.

As we have mentioned, we really want to change \emph{injective} in Corollary \ref{c:SD_injective} to \emph{isomorphism} by computing the number of sections of determinant bundles on the moduli spaces $M_S(e)$. These moduli spaces are more mysterious than the Hilbert scheme of points, but we imagine that there could be a way to generalize Theorem \ref{thm:chi_gen} beyond Hilbert schemes.

The Grothendieck quot scheme argument from \cite{marian_level-rank_2007} used in this paper does not generalize immediately to other (not del Pezzo) surfaces, although versions of strange duality are expected to hold (for strange duality results on a variety of surfaces, see \cite{danila_resultats_2002}, \cite{marian_sheaves_2008},
\cite{marian_tour_2008},
\cite{MR3079253},
\cite{marian_strange_2014}, \cite{MR2989230}, \cite{yuan_strange_2016}).
For example, one can check that on a $K$-trivial surface, candidates for strange duality \emph{never} yield quot schemes that are finite and reduced. However, if the points of the quot scheme still correspond to sections that span the space of sections of the theta bundle, then it may be possible to pick out a basis, perhaps as a Chern class of some excess intersection bundle on the quot scheme.

\bibliographystyle{alpha}
\bibliography{My_Library}{}

\end{document}